\documentclass[a4paper,12pt]{amsart}

\usepackage{hyperref}
\hypersetup{
    colorlinks=true,
    linkcolor=dodger,
    filecolor=dodger,
    urlcolor=dodger,
    citecolor=dodger,
}

\usepackage[top=3cm,left=2.5cm,right=2.5cm,bottom=3cm]{geometry}
\usepackage{relsize}
\usepackage{lineno}

\usepackage{amsmath,amsthm,pb-diagram,amssymb,comment}
\usepackage{amsfonts,graphicx,color}
\usepackage{enumitem, fancyhdr, dsfont, multicol}
\usepackage{thmtools,cleveref}
\usepackage[normalem]{ulem}
\usepackage{stmaryrd}
\usepackage{textcomp}
\usepackage{mathabx}
\usepackage{mathtools}
\usepackage{fontawesome}
\usepackage{textcomp}
\usepackage{mathabx}
\usepackage{tikz,float}
\usepackage{mleftright}

\setlist{topsep=0ex,itemsep=1ex}

    \newcommand{\thzfc}{\mathrm{ZFC}}

    \newcommand{\Bwf}{\mathcal{B}}
    \newcommand{\Cwf}{\mathcal{C}}

    \newcommand{\Hwf}{\mathcal{H}}
    \newcommand{\Iwf}{\mathcal{I}}
    \newcommand{\Jwf}{\mathcal{J}}
    
    \newcommand{\Mwf}{\mathcal{M}}
    \newcommand{\Nwf}{\mathcal{N}}
    
    \newcommand{\Pwf}{\mathcal{P}}

    \newcommand{\bfrak}{\mathfrak{b}}
    \newcommand{\cfrak}{\mathfrak{c}}
    \newcommand{\dfrak}{\mathfrak{d}}

    \newcommand{\menos}{\smallsetminus}
    
    \DeclareMathOperator{\pts}{\mathcal{P}}
    
    \newcommand{\frestr}{{\upharpoonright}}
    \DeclareMathOperator{\add}{{\rm add}}
    \DeclareMathOperator{\cov}{{\rm cov}}
    \DeclareMathOperator{\non}{{\rm non}}
    \DeclareMathOperator{\cof}{{\rm cof}}

    \newcommand{\Aor}{\mathbb{A}}
    
    \newcommand{\Cor}{\mathbb{C}}
    \newcommand{\Dor}{\mathbb{D}}
    \newcommand{\Eor}{\mathbb{E}}
    \newcommand{\Hor}{\mathbb{H}}

    \newcommand{\Por}{\mathbb{P}}
    
    \newcommand{\Qor}{\mathbb{Q}}

    \newcommand{\Qnm}{\dot{\mathbb{Q}}}

    \newcommand{\SNwf}{\mathcal{SN}}

    \newcommand{\Q}{\mathbb{Q}}
    \newcommand{\R}{\mathbb{R}}

    \DeclareMathOperator{\cf}{{\rm cf}}
    \DeclareMathOperator{\scf}{{\rm scf}}

    \newcommand{\imp}{\mathrel{\mbox{$\Rightarrow$}}}

    \newcommand{\la}{\langle}
    \newcommand{\ra}{\rangle}

\newcommand{\tbf}{\mathbf{t}}

\newcommand{\Rbf}{\mathbf{R}}

\DeclareMathOperator{\Lb}{\mathbf{Lb}}

\DeclareMathOperator{\hgt}{\mathrm{ht}}

\newcommand{\leqT}{\mathrel{\mbox{$\preceq_{\mathrm{T}}$}}}
\newcommand{\eqT}{\mathrel{\mbox{$\cong_{\mathrm{T}}$}}}
\newcommand{\supcof}{\mathrm{supcof}}
\newcommand{\supcov}{\mathrm{supcov}}
\newcommand{\minadd}{\mathrm{minadd}}
\newcommand{\minnon}{\mathrm{minnon}}

\newcommand{\baire}{\omega^\omega}

\newcommand{\Cbf}{\mathbf{C}}

\newcommand{\set}[2]{\left\{#1 :\, #2\right\}}
\newcommand{\Seq}[2]{\la #1 :\, #2\ra}

\newcommand{\Fn}{\mathrm{Fn}}

\newcommand{\DS}{\mathbf{DS}}
\newcommand{\bd}{\mathrm{bd}}
\newcommand{\baireincr}{\omega^{{\uparrow}\omega}}

\DeclareMathOperator{\ir}{\mathrm{tl}}

\definecolor{sub0}{RGB}{29,32,137}
\definecolor{sub1}{RGB}{1,71,157}
\definecolor{sub2}{RGB}{1,104,183}
\definecolor{sub3}{RGB}{0,160,234}
\definecolor{sug}{RGB}{0,154,68}
\definecolor{suy}{RGB}{208,219,1}

\definecolor{dodger}{rgb}{0.0,0.5,1.0}
\definecolor{carrotorange}{rgb}{0.93, 0.57, 0.13}

\title[The cofinality and covering of $\SNwf$]{More about the cofinality and the covering of the ideal of strong measure zero sets}

\author{Miguel A.~Cardona}
\address{Einstein Institute of Mathematics\\
Edmond J. Safra Campus, Givat Ram\\
The Hebrew University of Jerusalem\\
Jerusalem, 91904, Israel}
\email{miguel.cardona@mail.huji.ac.il}
\urladdr{https://sites.google.com/mail.huji.ac.il/miguel-cardona-montoya/home-page}

\author{Diego A.~Mej\'ia}
\address{Graduate School of System Informatics, Kobe University, 1-1 Rokkodai-cho, Nada-ku, Kobe, Hyogo 657--8501 Japan}
\email{damejiag@people.kobe-u.ac.jp}
\urladdr{http://www.researchgate.net/profile/Diego\_Mejia2}

\thanks{The first author was partially supported by the Slovak Research and Development Agency under Contract No.~APVV-20-0045 and by Pavol Jozef \v{S}af\'arik University at a postdoctoral position; and the second author was supported by the Grants-in-Aid for Scientific Research (C) 23K03198, Japan Society for the Promotion of Science.}



\begin{document}

\makeatletter
\def\@roman#1{\romannumeral #1}
\makeatother

\newcounter{enuAlph}
\renewcommand{\theenuAlph}{\Alph{enuAlph}}

\numberwithin{equation}{section}
\renewcommand{\theequation}{\thesection.\arabic{equation}}

\theoremstyle{plain}
  \newtheorem{theorem}[equation]{Theorem}
  \newtheorem{corollary}[equation]{Corollary}
  \newtheorem{lemma}[equation]{Lemma}
  \newtheorem{mainlemma}[equation]{Main Lemma}
  \newtheorem{prop}[equation]{Proposition}
  \newtheorem{clm}[equation]{Claim}
  \newtheorem{fact}[equation]{Fact}
  \newtheorem{exer}[equation]{Exercise}
  \newtheorem{question}[equation]{Question}
  \newtheorem{problem}[equation]{Problem}
  \newtheorem{conjecture}[equation]{Conjecture}
  \newtheorem{assumption}[equation]{Assumption}
  \newtheorem*{thm}{Theorem}
  \newtheorem{teorema}[enuAlph]{Theorem}
  \newtheorem*{corolario}{Corollary}
\theoremstyle{definition}
  \newtheorem{definition}[equation]{Definition}
  \newtheorem{example}[equation]{Example}
  \newtheorem{remark}[equation]{Remark}
  \newtheorem{notation}[equation]{Notation}
  \newtheorem{context}[equation]{Context}

  \newtheorem*{defi}{Definition}
  \newtheorem*{acknowledgements}{Acknowledgements}

\def\sectionautorefname{Section}
\def\subsectionautorefname{Subsection}


\begin{abstract}

We improve the previous work of Yorioka and the first author about the combinatorics of the ideal $\mathcal{SN}$ of strong measure zero sets of reals. We refine the notions of dominating systems of the first author and introduce the new combinatorial principle $\mathbf{DS}(\delta)$ that helps to find simple conditions to deduce $\mathfrak{d}_\kappa \leq \mathrm{cof}(\mathcal{SN})$ (where $\mathfrak{d}_\kappa$ is the dominating number on $\kappa^\kappa$). In addition, we find a new upper bound of $\mathrm{cof}(\mathcal{SN})$ by using products of relational systems and cardinal characteristics associated with Yorioka ideals. 

In addition, we dissect and generalize results from Pawlikowski to force upper bounds of the covering of $\mathcal{SN}$, particularly for finite support iterations of precaliber posets.

Finally, as applications of our main theorems, we prove consistency results about the cardinal characteristics associated with $\mathcal{SN}$ and the principle $\mathbf{DS}(\delta)$. For example, we show that $\mathrm{cov}(\mathcal{SN})<\mathrm{non}(\mathcal{SN})=\mathfrak{c}<\mathrm{cof}(\mathcal{SN})$ holds in Cohen model, and we refine a result (and the proof) of the first author about the consistency of $\mathrm{cov}(\mathcal{SN})<\mathrm{non}(\mathcal{SN})<\mathrm{cof}(\mathcal{SN})$, with $\mathfrak{c}$ in any desired position with respect to $\mathrm{cof}(\mathcal{SN})$, and the improvement that $\mathrm{non}(\mathcal{SN})$ can be singular here.
\end{abstract}

\maketitle

\section{Introduction}\label{SecIntro}


The goal of this work is to further study the cofinality and covering of the $\sigma$-ideal $\SNwf$ of the strong measure zero sets of reals. We continue the line of investigation from~\cite{P90,Yorioka,cardona} and strengthen their main results. 

Before entering into details, we review some basic notions and notation.

\begin{notation}
    \ 
    \begin{enumerate}[label = (\arabic*)]
        \item Given a formula $\phi$, $\forall^\infty\, n<\omega\colon \phi$ means that all but finitely many natural numbers satisfy $\phi$; $\exists^\infty\, n<\omega\colon \phi$ means that infinitely many natural numbers satisfy $\phi$.

        \item For $x,y\in \omega^\omega$, $x\leq^* y$ means that $x(n)\leq y(n)$ for all but finitely many $n<\omega$. A \emph{dominating family in $\omega^\omega$} is a subset $D\subseteq\omega^\omega$ such that, for any $x\in\omega^\omega$, $x\leq^* y$ for some $y\in D$. The \emph{dominating number $\dfrak$} is the smallest size of a dominating family.

        \item $\omega^{\uparrow\omega}$ denotes the set of all increasing functions in $\omega^\omega$.

        \item $\cfrak:=2^{\aleph_0}$; $\Nwf$ and $\Mwf$ denote the ideals of Lebesgue measure zero sets and of meager sets in the Cantor space  $2^\omega$, respectively.

        \item For $s\in 2^{<\omega}$, denote $[s]:=\{x\in 2^\omega:\, s\subseteq x\}$.

        \item For $\sigma \in (2^{<\omega})^\omega$, define $\hgt_\sigma\colon \omega\to\omega$ by $\hgt_\sigma(n):=|\sigma(n)|$ for all $n<\omega$, which we call the \emph{height of $\sigma$}. Also, define
        \[[\sigma]_\infty:=\{x\in 2^\omega:\, \exists^\infty n: \sigma(n)\subseteq x\}.\]
    \end{enumerate}
\end{notation}

\begin{definition}
A set \emph{$X\subseteq 2^\omega$ has strong measure zero} if
\[\forall\, f\in\omega^\omega\ \exists\, \sigma\in(2^{<\omega})^\omega\colon f\leq^* \hgt_\sigma \text{ and }X\subseteq\bigcup_{i<\omega}[\sigma(i)].\]
%
Denote by $\SNwf(2^\omega)$ the collection of strong measure zero sets of $2^\omega$.
\end{definition}

The following characterization of $\SNwf(2^\omega)$ is quite practical.

\begin{lemma}\label{charSN}
    Let $X\subseteq2^\omega$ and let $D\subseteq\omega^\omega$ be a dominating family. Then $X\subseteq2^\omega$ has strong measure zero in $2^\omega$ iff  
    \[\forall f\in D\ \exists\sigma\in (2^{<\omega})^\omega\colon f\leq^*\hgt_\sigma\textrm{\ and\ } X\subseteq[\sigma]_\infty.\]
\end{lemma}


Let $\Iwf$ be an ideal of subsets of $X$ such that $\{x\}\in \Iwf$ for all $x\in X$. We define \emph{the cardinal characteristics associated with $\Iwf$} by
\begin{align*}
 \add(\Iwf)&=\min\{|\Jwf|:\,\Jwf\subseteq\Iwf,\,\bigcup\Jwf\notin\Iwf\}\\
 \cov(\Iwf)&=\min\{|\Jwf|:\,\Jwf\subseteq\Iwf,\,\bigcup\Jwf=X\}\\
 \non(\Iwf)&=\min\{|A|:\,A\subseteq X,\,A\notin\Iwf\}\\
 \cof(\Iwf)&=\min\{|\Jwf|:\,\Jwf\subseteq\Iwf,\ \forall\, A\in\Iwf\ \exists\, B\in \Jwf\colon A\subseteq B\}.
\end{align*}
These cardinals are referred to as \emph{additivity, covering, uniformity} and \emph{cofinality of $\Iwf$}, respectively.  The relationship between the cardinals defined above is illustrated by \autoref{diag:idealI}. 

\begin{figure}[h]
  \begin{center}
    \includegraphics[scale=1.0]{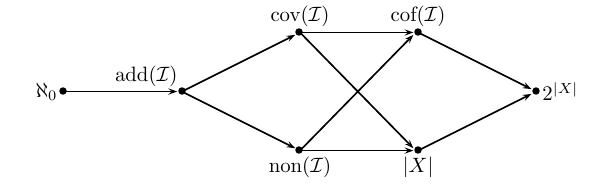}
    \caption{Diagram of the cardinal characteristics associated with $\Iwf$. An arrow  $\mathfrak x\rightarrow\mathfrak y$ means that (provably in ZFC) 
    $\mathfrak x\le\mathfrak y$.}
    \label{diag:idealI}
  \end{center}
\end{figure}

The cardinal characteristics associated with the strong measure zero ideal are the same for the spaces $2^\omega$, $\R$ and $[0,1]$, 
see details in e.g.~\cite{CMR}. From now on, we work with $\SNwf=\SNwf(2^\omega)$.

A number of
people, e.g.~\cite{Mi1982,JS89,P90,GJS,Yorioka,O08,CMR,cardona}, has been studying for decades the relation between the cardinal characteristics associated with $\SNwf$ and other classical cardinal characteristics of the continuum. Most of them are illustrated in \autoref{Cichonwith_SN}, with the exception of the closest upper bound of $\cof(\SNwf)$, which is one of the main results of this work. The cited references also show that the diagram is quite complete, but some open questions remain, for instance, it is not known whether $\bfrak$, $\dfrak$, $\non(\Nwf)$, $\supcof$ and $\cof(\Nwf)$ are lower bounds of $\cof(\SNwf)$, and whether $\add(\Nwf)=\minadd$ and $\cof(\Nwf)=\supcof$.


\begin{figure}[ht!]
\begin{center}
  \includegraphics[scale=0.84]{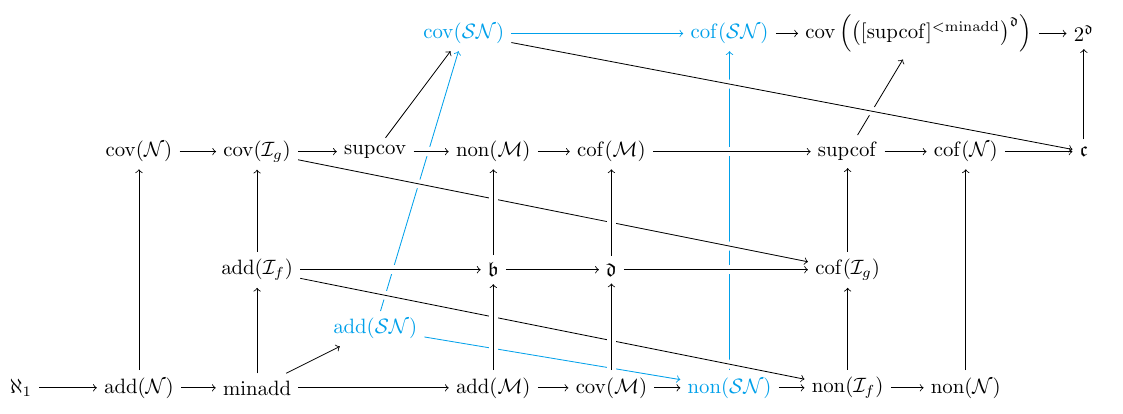}
 \caption{Cicho\'n's diagram with the cardinal characteristics associated with $\SNwf$ and $\Iwf_f$. Moreover, $\add(\Mwf)=\min\{\bfrak,\non(\SNwf)\}$ and $\cof(\Mwf)=\max\{\dfrak,\supcov\}$. The inequality $\cof(\SNwf)\leq\cov\left(\left([\supcof]^{<\minadd}\right)^\dfrak\right)$ is original from this work (\autoref{new_upperbI}).}
  \label{Cichonwith_SN}
\end{center}
\end{figure}

Regarding the cofinality of $\SNwf$, it is known that CH implies $\cof(\SNwf)>\cfrak$.
Some of the earliest results are due to J.~Cicho\'n who proved that $\cof(\SNwf)\leq2^{\dfrak}$ holds in ZFC and, as a consequence, that $\cof(\SNwf)=\aleph_2$ follows from GCH (see~\cite[Cor.~3.3 \&~3.6]{Ser}). 

In 2002, Yorioka~\cite{Yorioka} introduced a characterization of $\SNwf$ in terms of the $\sigma$-ideals $\Iwf_f$ parametrized by $f\in\omega^{\uparrow\omega}$, 
which are called~\emph{Yorioka ideals}. 
The definition of these ideals seems to have been inspired by~\autoref{charSN}:

\begin{definition}[Yorioka~{\cite{Yorioka}}]\label{DefYorio}
For $f\in\omega^{\uparrow\omega}$ define the \emph{Yorioka ideal}
\[\Iwf_f:=\{A\subseteq 2^\omega:\, \exists\, \sigma\in(2^{<\omega})^\omega\colon f\ll \hgt_\sigma\text{\ and }A\subseteq[\sigma]_\infty\}\]
where the relation $x \ll y$  denotes $\forall\, k<\omega\ \forall^\infty\, i<\omega\colon x(i^k)\leq y(i)$.

Define $\minadd:=\min\set{\add(\Iwf_f)}{f\in\omega^{\uparrow\omega}}$ and $\supcof:=\sup\set{\cof(\Iwf_f)}{ f\in\omega^{\uparrow\omega}}$. The cardinal characteristics $\supcov$ and $\minnon$ are defined analogously. Recall that $\minnon = \non(\SNwf)$.
\end{definition}

The reason why $\ll$ is used instead of $\leq^*$ is that the latter would not yield an ideal, as it was proved by Kamo and Osuga~\cite{KO08}.
Yorioka~\cite{Yorioka} has proved (indeed) that $\Iwf_f$ is a $\sigma$-ideal when $f$ is increasing. By~\autoref{charSN} it is clear that $\SNwf=\bigcap\{\Iwf_f:\, f\in \omega^{\uparrow\omega}\}$.

These ideals led Yorioka to rediscover and greatly improve the results from~\cite{Ser}, and to show that no inequality between $\cof(\SNwf)$ and $\cfrak$ can be decided in ZFC. To achieve this, he established the following characterization of $\cof(\SNwf)$ in terms of the dominating number of $\lambda^\lambda$, denoted by $\dfrak_\lambda$ (its definition is presented in~\autoref{exm:dirpow}).
The relation $\leq$ denotes pointwise $\leq$ (everywhere).

\begin{theorem}[Yorioka {\cite{Yorioka}}]\label{Ycof}
If $\minadd=\supcof=\lambda$ then $\SNwf$ and $\la\lambda^\lambda,\leq\ra$ are Tukey equivalent. 
In particular, $\add(\SNwf)=\lambda$ and $\cof(\SNwf)=\dfrak_\lambda$.\footnote{Yorioka's original theorem contains the additional hypothesis $\cov(\Mwf)=\dfrak=\lambda$, but it is now known that it follows from $\minadd=\supcof$ (see \autoref{Cichonwith_SN}). However, the equality $\cov(\Mwf)=\dfrak$ is relevant for the generalizations of Yorioka's Theorem in~\cite{cardona} and in this work.}
\end{theorem}

Inspired by Yorioka's proof of~\autoref{Ycof}, the first author~\cite{cardona} introduced the notion of~\emph{$\bar f$-dominating directed system} for a directed preorder $S$ (see~\autoref{def:I_fdirsystS}) where $\bar f =\la f_\alpha:\, \alpha<\lambda\ra$ forms a dominating family of increasing functions. This allowed providing bounds for $\cof(\SNwf)$ with hypotheses weaker than $\minadd=\supcof$, which led to a generalization of~\autoref{Ycof}. 
See \autoref{examSdir} for the definition of $\dfrak^\lambda_S$.

\begin{theorem}[{\cite[Thm.~3.8~and~Cor.~3.13]{cardona}}]\label{thmcar} Let $S$ be a directed preorder. 
 \begin{enumerate}[label=\rm(\arabic*)]
     \item\label{thmcar1} Assume that there is some $\bar f$-dominating directed system on 
$S$ with $\bar f=\la f_\alpha:\, \alpha<\delta\ra$. Then $\SNwf$ is Tukey below $S^\delta$. In particular, $\bfrak(S)\leq\add(\SNwf)$ and $\cof(\SNwf)\leq\dfrak^\delta_S$. 

     \item\label{thmcar2} Assume that $\kappa$ and $\lambda$ are cardinals such that $0<\kappa\leq\lambda\leq\non(\SNwf)$ and $\bar f=\la f_\alpha:\, \alpha<\lambda\ra$ is dominating in $\baireincr$. If there is some $\bar f$-dominating directed system on $\kappa\times\lambda$ then
  $\la\lambda^\lambda,\leq\ra$ is Tukey below $\SNwf$, in particular $\add(\SNwf)\leq\cf(\lambda)$ and $\dfrak_\lambda\leq\cof(\SNwf)$. 
 \end{enumerate}   
\end{theorem}
 
One of the main purposes of the present paper is to improve \autoref{thmcar} (and Yorioka's \autoref{Ycof} as well) to get better upper and lower bounds of $\cof(\SNwf)$ with much simpler hypotheses. We achieve this by simplifying the first author's notion of $\bar f$-dominating directed system while generalizing many results from~\cite{cardona}. 
The new notion of \emph{$\bar f$-dominating system} (\autoref{def:I_fsystS})
gets rid of the directed preorders as a requirement, which helps to considerably strengthen the previously cited theorems, while discovering new theorems to construct uncountable strong measure zero sets. 

Concretely, to deal with upper bounds of the cofinality of $\SNwf$, 
we present a theory of products and quotients for relational systems in \autoref{Sec:Tukey} that lets us improve the inequality $\cof(\SNwf)\leq 2^\dfrak$ in ZFC alone. 
%
%
 %
%
See concrete notation in \autoref{def:idpow}.

\begin{teorema}[\autoref{new_upperb}]\label{new_upperbI}
$\cof(\SNwf)\leq\cov\left(\left([\supcof]^{<\minadd}\right)^\dfrak\right).$
\end{teorema}

In addition,
we introduce a principle denoted by $\DS(\delta)$ for an ordinal $\delta$, which basically states that there is an $\bar f$-dominating system where $\bar f$ is a sequence of length $\delta$ that forms a dominating family of increasing functions (see \autoref{BounSN}). 
We use this principle to give a general criterion to construct uncountable strong measure zero sets.

\begin{teorema}[{\autoref{SNpower}}]\label{SNpower:I}
 Assume $\DS(\delta)$. 
 Let $\la J_\alpha:\, \alpha<\delta\ra$ be a sequence of sets and assume, for each $\alpha<\delta$,
\begin{enumerate}[label =\rm (H\arabic*)]
    \item $\{C^\alpha_j:\, j\in J_\alpha\}\subseteq \Iwf_{f_\alpha}$, and
    \item $\sum_{\alpha'<\alpha}|J_{\alpha'}|<\non(\SNwf)$.
\end{enumerate}
Then there 
is some set $K\in\SNwf$ of size $\sum_{\alpha<\delta}|J_\alpha|$ such that $K\nsubseteq C^\alpha_j$ for all $\alpha<\delta$ and $j\in J_\alpha$.
\end{teorema}

This result has deep consequences on $\cof(\SNwf)$ and gives general criteria to find lower bounds of $\cof(\SNwf)$, as a generalization of \autoref{Ycof} and \autoref{thmcar}~\ref{thmcar2}.


\begin{teorema}[{\autoref{non<cof}}]\label{non<cof:I}
Assume $\DS(\delta)$. 
\begin{enumerate}[label = \rm (\alph*)]
    \item If $\delta\leq\non(\SNwf)$ then $\delta<\cof(\SNwf)$.
    \item  If $\cf(\non(\SNwf))=\cf(\delta)$ then $\non(\SNwf)<\cof(\SNwf)$ and there is some $K\in\SNwf$ of size $\non(\SNwf)$.
\end{enumerate}      
\end{teorema}



\begin{teorema}[{\autoref{lowerSN}}]\label{lowerSN:I}
Assume $\DS(\delta)$ and $\non(\SNwf)=\supcof=\mu$.
If $\lambda:=\cf(\delta)=\cf(\mu)$,
then $\la \lambda^\lambda,\leq\ra$ is Tukey below $\SNwf$. In par\-ti\-cu\-lar $\add(\SNwf)\leq\lambda$ and $\dfrak_\lambda\leq\cof(\SNwf)$. Moreover, $\dfrak_\lambda\neq\mu$ and $\mu<\cof(\SNwf)$.
\end{teorema}

The previous theorem actually sheds more light on \autoref{thmcar}~\ref{thmcar2}: the hypothesis of this referred result implies that $\lambda$ is regular and $\dfrak=\non(\SNwf)=\cof(\SNwf)=\lambda$. This solves a question from~\cite[Section~5]{cardona} about the role of $\kappa$ in the hypothesis of \autoref{thmcar}~\ref{thmcar2}, which now seems to be a bit irrelevant.
See details at the end of \autoref{BounSN}.




Another main result of our work is about forcing an upper bound of $\cov(\SNwf)$ in forcing iterations. Judah and Shelah~\cite{JS89} have constructed a specific FS (finite support) iteration of precaliber $\aleph_1$ posets to prove the consistency of $\add(\SNwf)<\add(\Mwf)$. Motivated by this result, Pawlikowski~\cite{P90} proved that certain iterations of posets with precaliber $\aleph_1$ (particularly FS iterations), with length of uncountable cofinality, force $\cov(\SNwf)=\aleph_1$.





In this paper, we dissect Pawlikowski's proof to reduce the proof to very simple facts about iterations, Cohen reals, and strong measure zero sets, allowing us to generalize his result and find conditions to force $\cov(\SNwf)\leq\theta$ for a given regular cardinal $\theta$ through certain iterations of precaliber $\theta$ posets (including FS iterations). We isolate the main notion of this proof, which we call \emph{$\theta$-Rothberger sequences}, and show how they produce strong measure zero sets and coverings of the reals from such sets. 

\begin{teorema}[\autoref{thm:precaliberstr}]\label{thm:precaliberI}
Assume that $\theta>\aleph_0$ is regular. 
Let $\pi$ be a limit ordinal with $\cf(\pi)>\aleph_0$ and let
$\tbf=\la \Por_\xi:\, \xi\leq\pi\ra$ be a sequence of posets satisfying: 
\begin{enumerate}[label = \rm(\roman*)]
    \item $\Por_\xi$ is a complete subset of $\Por_\eta$ for $\xi\leq\eta\leq\pi$,
    \item $\Por_\pi=\bigcup_{\xi<\pi}\Por_\xi$ (direct limit) and it has $\cf(\pi)$-cc,
    \item $\Por_\eta=\bigcup_{\xi<\eta}\Por_\xi$ for any $\eta\leq\pi$ with $\cf(\eta)=\theta$, and it has $\theta$-cc,
    \item $\Por_\eta$ forces that $\Por_\pi/\Por_\eta$ has precaliber $\theta$ whenever $\eta\leq\pi$ with $\cf(\eta)=\theta$, and
    \item for any $\xi<\pi$, $\Por_{\xi+1}$ adds a Cohen real over $V^{\Por_\xi}$.
\end{enumerate}
Let $\lambda$ be the largest size of a partition of $\pi$ into cofinal subsets. If $\lambda\geq \theta$ then $\Por_\pi$ forces $\cov(\SNwf)\leq\theta$ and $\lambda\leq\non(\SNwf)$.
\end{teorema}

We also present applications of our results in forcing extensions. The first author~\cite{cardona} proved the consistency with ZFC of
\[\add(\SNwf)=\cov(\SNwf)<\non(\SNwf)<\cof(\SNwf).\]
On the other hand, our results about the influence of the principle $\DS(\delta)$ to $\cof(\SNwf)$ imply that the above is already true in Cohen model. However, in Cohen model we get $\cof(\SNwf)>\cfrak$, while the model of the first author allows forcing $\cfrak$ in any position with respect to $\cof(\SNwf)$ (i.e. smaller, equal, or larger).

Our work not only allows shortening the first author's proof of the result above, but we have one important improvement: while the original models restrict $\non(\SNwf)$ to be regular, we use the technique of matrix iterations with vertical support restrictions from the second author~\cite{mejiavert} to improve the model and obtain that $\non(\SNwf)$ may be singular.

\begin{teorema}[\autoref{Cmainth}]\label{mainappl}
    Assume that $\kappa\leq\lambda\leq\lambda_1$ and $\lambda_2$ are uncountable cardinals such that $\kappa$ is regular, $\cof([\lambda]^{<\kappa})=\lambda$, $\cf(\lambda)^{<\cf(\lambda)}=\cf(\lambda)$, $\lambda_1=\lambda_1^{\aleph_0}$ and $\lambda_2^\lambda=\lambda$.\footnote{Note that no relation between $\lambda_1$ and $\lambda_2$ is assumed, so any of the three possibilities $\lambda_1<\lambda_2$, $\lambda_1=\lambda_2$ and $\lambda_2<\lambda_1$ can be assumed.} Then there is a cofinality preserving poset forcing
    \[\add(\SNwf)= \cov(\SNwf)=\kappa,\ \non(\SNwf) = \lambda,\ \cfrak=\lambda_1 \text{ and } \cof(\SNwf)=\lambda_2.\]
\end{teorema}

We also show that $\DS(\delta)$ for some $\delta$ is always forced by FS iterations, and show some standard models where this principle holds. We do not know whether it is consistent with ZFC that $\DS(\delta)$ is false for all $\delta$. There could be some chance that this is connected with the Borel Conjecture, and maybe this is true in Laver or Mathias models. 


\noindent\textbf{Structure of the paper.} 
In \autoref{sec:coft} we present some results in ordinal arithmetic that will be useful for properly stating some of the main results of this work. In~\autoref{Sec:Tukey}, we review all the material related to relational systems and the Tukey order, and we present our theory of products and quotients of relational systems, with relevant examples for this work. This theory allows the generalization of some results of the first author~\cite[Sect.~2]{cardona} and Brendle~\cite{brehig}.

In \autoref{BounSN} we introduce our new notion of $\bar f$-dominating system, the principles $\DS(\delta)$, analyze its combinatorics, and prove 
\autoref{new_upperbI}--\ref{lowerSN:I}. For our results about $\cov(\SNwf)$, \autoref{CovSN} is dedicated to the study of $\theta$-Rothberger sequences and to prove~\autoref{thm:precaliberI}. The applications of the previous results are presented in~\autoref{Sec:eff}, where we construct models forcing principles of the form $\DS(\delta)$ and prove \autoref{mainappl}. 
Finally, in~\autoref{sec:Oq} we address some open questions about this work.

\section{Cofinal types of ordinal numbers}\label{sec:coft}

We now provide some basic preliminary results about ordinal arithmetic that are needed in the coming sections. Recall that an ordinal $\pi$ is~\emph{(additively) indecomposable} if $\xi+\eta<\pi$ for all $\xi<\pi$ and $\eta<\pi$. It is well-known that, when $\pi>0$, indecomposable is equivalent to $\pi=\omega^\varepsilon$ (ordinal exponentiation) for some ordinal $\varepsilon$.

\begin{definition}\label{def:cardtail}
    Let $\pi>0$ be an ordinal. 
    \begin{enumerate}[label=\rm(\arabic*)]
        \item Define the \emph{segment cofinality of $\pi$} by
    \[\scf(\pi):=\min\set{|c|}{c\subseteq \pi \text{ is a non-empty final segment of }\pi}.\]
        \item Recall that $\pi$ can be expressed in \emph{Cantor normal form}
        \[\pi=\omega^{\varepsilon_0}+\cdots + \omega^{\varepsilon_n}\]
        for some unique $0<n<\omega$ and unique ordinals $\varepsilon_0\geq\cdots \geq\varepsilon_n$. We denote $\ir(\pi):= \omega^{\varepsilon_n}$, which we call \emph{the indecomposable tail of $\pi$}.
    \end{enumerate}    
\end{definition}

Note that $\ir(\pi)$ is the smallest ordinal $\eta>0$ such that $\xi+\eta = \pi$ for some $\xi<\pi$. Using the Cantor normal form above, we can show that there are only finitely many order types for the final segments of $\pi$, and they are of the form $\omega^{\varepsilon_m}+\cdots + \omega^{\varepsilon_n}$ for any $m\leq n$ (so there are $n+1$-many such order types). See details in~\cite{Hamkins}. As a consequence:

\begin{lemma}\label{clmfscf}
    If $\pi>0$ is an ordinal, then $\scf(\pi)=|\ir(\pi)|$. In particular, $\scf(\pi)=|\pi|$ whenever $\pi$ is indecomposable.
\end{lemma}

The segment cofinality has the following very interesting characterization.

\begin{lemma}\label{lem:scf}
    Let $\pi>0$ be an ordinal. Then
    \[\scf(\pi)=\max\set{|C|}{\text{$C$ is a pairwise disjoint family of cofinal subsets of $\pi$}}.\]   
\end{lemma}
\begin{proof}
We show that  
\begin{enumerate}[label=\rm (\alph*)]
    \item\label{lem:scf1} $\scf(\pi)\in\set{|C|}{\text{$C$ is a pairwise disjoint family of cofinal subsets of $\pi$}}$ and 

    \item\label{lem:scf2} for any pairwise disjoint family $C$ of cofinal subsets of $\pi$, $|C|\leq\scf(\pi)$.
\end{enumerate}
To see~\ref{lem:scf1}, let $\lambda:=\scf(\pi)$ and choose a final segment $[\xi,\pi)$ of size $\lambda$ for some $\xi<\pi$. It suffices to construct a pairwise disjoint family $\set{C_\beta}{\beta<\lambda}$ of cofinal subsets of $\pi$. Enumerate $[\xi,\pi)$ as $\set{\xi_\alpha}{\alpha<\lambda}$ and let $f = (f_0,f_1):\lambda\to\lambda\times\lambda$ be a bijection.\footnote{It is possible that $\lambda$ is finite, in which case $\lambda=1$.} Then, by recursion, construct $g\colon\lambda\to [\xi,\pi)$ such that, for all $\alpha<\lambda$,
\begin{enumerate}[label=\rm (\roman*)]

    \item $g(\alpha)\in[\xi,\pi)\smallsetminus\set{g(\beta)}{\beta<\alpha}$, and 

    \item $g(\alpha)\geq \xi_{f_1(\alpha)}$. 
\end{enumerate}
For $\beta<\lambda$, set $C_\beta:=\set{g(\alpha)}{\alpha<\lambda,\  f_0(\alpha)=\beta}$. It is not hard to see that $\set{C_\beta}{\beta<\lambda}$ is the required family.

For~\ref{lem:scf2}, assume that $C$ is a pairwise disjoint family of cofinal subsets of $\pi$ and let $c'$ be a witness of $\scf(\pi)$. Next, for $a\in C$, since $a$ is cofinal in $\pi$, we 
have $a\cap c'\neq \emptyset$.
It is clear that $\set{a\cap c'}{a\in C}$ is a pairwise disjoint family of nonempty subsets of $c'$, which implies that $|C|\leq|c'|$, so we are done. 
\end{proof}

We now aim to classify the cofinal types of an ordinal $\pi>0$, that is, the order types of all the cofinal subsets of $\pi$, and we use the notion of an indecomposable tail for this purpose. It is clear that $1$ is the only cofinal type of any successor ordinal. In the case of indecomposable ordinals, the cofinal types are easy to classify.\footnote{An alternative proof appears in~\cite{MOf}.}

\begin{lemma}\label{cftype:indec}
    Let $\beta\leq\gamma$ be ordinals such that $\gamma$ is indecomposable. Then, there is some cofinal $C\subseteq\gamma$ of order type $\beta$ iff $\cf(\beta)=\cf(\gamma)$.
\end{lemma}
\begin{proof}
    The implication from left to right is obvious, so we prove the converse.
    This is trivially true when $\beta=\gamma$, so assume that $\beta<\gamma$ and $\cf(\beta)=\cf(\gamma)$. Then, $\gamma$ must be a limit ordinal (and so $\beta$ too), because otherwise $\gamma=1$ and $\beta=0$, which have different cofinalities.
    It suffices to find an increasing and cofinal function $f\colon\beta\to\gamma$. 
    The idea to construct $f$ is to pair segments of $\beta$ with segments of $\gamma$ in a natural way. This pairing will be obtained by transfinite recursion.

    Let $\kappa:=\cf(\beta)=\cf(\gamma)$, choose a increasing continuous sequence $\Seq{\beta_i}{i<\kappa}$, cofinal in $\beta$ with $\beta_0=0$, and choose a cofinal sequence $\Seq{\gamma_i}{i<\kappa}$ in $\gamma$. By recursion on $i<\kappa$, construct ordinals $\gamma^-_i$ and $\gamma^+_i$ such that
    \begin{enumerate}[label = \rm(\roman*)]
        \item $\gamma^-_i<\gamma^+_i <\gamma^-_j<\gamma$ for $i<j<\kappa$,
        \item\label{gamma2} $\gamma^-_i > \gamma_i$, and
        \item\label{gamma3} $\gamma^+_i = \gamma^-_i + (\beta_{i+1}-\beta_i)$, so the interval $[\gamma^-_i,\gamma^+_i)$ has order type $\beta_{i+1}-\beta_i$.
    \end{enumerate}
    This construction is possible because $\cf(\gamma)=\kappa$, and $\gamma^+_i<\gamma$ because $\gamma$ is indecomposable.

    Now, for each $i<\kappa$, by~\ref{gamma3} there is an increasing bijection $f_i\colon [\beta_i,\beta_{i+1})\to[\gamma^-_i,\gamma^+_i)$. Let $f:=\bigcup_{i<\kappa} f_i$. It is clear that $f\colon \beta\to\gamma$ is increasing, and it is cofinal by~\ref{gamma2}.
    %
\end{proof}

\begin{theorem}\label{cftype}
    Let $\pi>0$ be an ordinal. For $\alpha\leq\pi$, the following are equivalent:
    \begin{enumerate}[label=\rm(\roman*)]
        \item\label{cftypea} There is some cofinal $C\subseteq \pi$ of order type $\alpha$.
        \item\label{cftypeb} $\alpha=\xi+\eta$ for some $\xi<\pi$ and $\eta\leq \ir(\pi)$ such that $\cf(\eta)=\cf(\pi)$.
    \end{enumerate}
\end{theorem}
\begin{proof}
\noindent\ref{cftypea}${}\Rightarrow{}$\ref{cftypeb}: Let $C\subseteq\pi$ be cofinal of order type $\alpha$. Pick a $\beta<\pi$ such that $\beta+\ir(\pi)=\pi$, let $\xi$ be the order type of $C\cap\beta$ and let $\eta$ be the order type of $C\smallsetminus\beta$. Then $\cf(\eta)=\cf(\pi)$ and $\xi<\pi$, and it is clear that $\alpha=\xi+\eta$.

\noindent\ref{cftypeb}${}\Rightarrow{}$\ref{cftypea}: Assume that there are $\xi<\pi$ and a $\eta\leq \ir(\pi)$ with $\cf(\eta)=\cf(\pi)$ such that $\alpha=\xi+\eta$. Pick some $\xi^+<\pi$ such that $\xi^+\geq\xi$ and $\xi^++\ir(\pi)=\pi$. Since $\eta\leq \ir(\pi)$,  by \autoref{cftype:indec} we can find a cofinal set $A\subseteq[\xi^+,\pi)$ of order type $\eta$. Finally, let $C:=\xi\cup A$, so it follows that $C\subseteq\pi$ is cofinal of order type $\alpha$. 
\end{proof}

We also classify the order types of cofinal subsets of a limit ordinal composed by limit ordinals. Note that limit ordinals with $\ir(\pi)=\omega$ cannot contain such cofinal subsets.

\begin{corollary}\label{cor:cflimtype}
    Let $\pi$ be a limit ordinal. For $\alpha\leq\pi$, the following are equivalent:
    \begin{enumerate}[label=\rm(\roman*)]
        \item\label{cflimti} There is some cofinal $C\subseteq \pi$ of order type $\alpha$, composed by limit ordinals.
        \item\label{cflimtii} $\alpha=\xi+\eta$ for some $\xi$ and $\eta$ such that $\omega\xi<\pi$, $\omega\eta\leq \ir(\pi)$ and $\cf(\eta)=\cf(\pi)$.
    \end{enumerate}
\end{corollary}
\begin{proof}
    Since $\pi$ is a limit ordinal, $\pi=\omega\gamma$ for some ordinal $\gamma>0$. Then, the limit ordinals smaller than $\pi$ have the form $\omega\beta$ for some $0<\beta<\gamma$. Note that~\ref{cflimti} is equivalent to:
    \begin{enumerate}[label = $(\bullet)_{\arabic*}$]
        \item\label{bli} There is some cofinal $C\subseteq \gamma$ of order type $\alpha$, and $\gamma$ is a limit ordinal.
    \end{enumerate}
    Therefore, by \autoref{cftype},~\ref{bli} is equivalent to
    \begin{enumerate}[resume*]
        \item\label{blii} $\alpha=\xi+\eta$ for some $\xi<\gamma$ and $\eta\leq\ir(\gamma)$ such that $\cf(\eta)=\cf(\gamma)$, and $\gamma$ is a limit ordinal.
    \end{enumerate}
    Since $\ir(\pi)=\omega\ir(\gamma)$, it is clear that~\ref{blii} and~\ref{cflimtii} are equivalent.
\end{proof}


We will use the following consequence for our forcing applications in \autoref{Sec:eff}.

\begin{corollary}\label{lemincseq}
    Assume that $\lambda$ is an uncountable cardinal and $\pi$ is a (positive) multiple of $\lambda$, i.e.\ $\pi= \lambda\gamma$ for some $\gamma>0$. Then the following statements are equivalent for $\beta\leq\gamma$:
    \begin{enumerate}[label = \rm(\roman*)]
        \item\label{lemincseqi} There is a cofinal subset of $\pi$ of order type $\lambda\beta$ composed by limit ordinals.
        \item\label{lemincseqii} $\beta=\delta+\rho$ for some $\delta<\gamma$ and $\rho\leq\ir(\gamma)$ such that $\cf(\lambda\rho)=\cf(\pi)$.
    \end{enumerate}
\end{corollary}
\begin{proof}
    By \autoref{cor:cflimtype},~\ref{lemincseqi} is equivalent to
    \begin{enumerate}[label = $(\star)$]
        \item\label{lemincseqiii} $\lambda\beta=\xi+\eta$ for some $\xi$ and $\eta$ such that $\omega\xi<\pi$, $\omega\eta\leq\ir(\pi)$ and $\cf(\eta)=\cf(\pi)$. 
    \end{enumerate}
    Assuming~\ref{lemincseqiii},
    since $\xi\leq \omega\xi<\pi$, $\xi=\lambda\delta+\varrho$ for some $\delta<\gamma$ and $\varrho<\lambda$.
    Also, $\lambda\beta=\xi+\eta$ implies that $\eta=\lambda\rho$ for some $\rho$, and since $\lambda=\omega^\lambda$ (because $\lambda$ is an uncountable cardinal), $\omega\eta=\lambda\rho\leq \ir(\pi)=\lambda\ir(\gamma)$, so $\lambda\beta=\lambda(\delta+\rho)$, $\rho\leq\ir(\gamma)$ and $\cf(\lambda\rho)=\cf(\pi)$. This implies~\ref{lemincseqii}.

    It is clear that~\ref{lemincseqii} implies~\ref{lemincseqiii}.
\end{proof}

\section{Relational systems, products and quotients}\label{Sec:Tukey}

In this section, we introduce several basic definitions concerning relational systems, the Tukey order, and cardinal characteristics. We also develop a general theory of products and quotients of relational systems.
We closely follow the presentation of~\cite[Sec.~1]{CM22}, which is based on~\cite{Vojtas,BartInv,blass}. 


\begin{definition}\label{def:relsys}
We say that $\Rbf=\la X, Y, \sqsubset\ra$ is a \textit{relational system} if it consists of two non-empty sets $X$ and $Y$ and a relation $\sqsubset$.
\begin{enumerate}[label=(\arabic*)]
    \item A set $F\subseteq X$ is \emph{$\Rbf$-bounded} if $\exists\, y\in Y\ \forall\, x\in F\colon x \sqsubset y$. 
    \item A set $D\subseteq Y$ is \emph{$\Rbf$-dominating} if $\forall\, x\in X\ \exists\, y\in D\colon x \sqsubset y$. 
\end{enumerate}

We associate two cardinal characteristics with the relational system $\Rbf$: 
\begin{itemize}
    \item[{}] $\bfrak(\Rbf):=\min\{|F|:\, F\subseteq X  \text{ is }\Rbf\text{-unbounded}\}$ the \emph{unbounding number of $\Rbf$}, and
    
    \item[{}] $\dfrak(\Rbf):=\min\{|D|:\, D\subseteq Y \text{ is } \Rbf\text{-dominating}\}$ the \emph{dominating number of $\Rbf$}.
\end{itemize}
\end{definition}

As in~\cite{CM22}, we also look at relational systems given by directed preorders.

\begin{definition}\label{examSdir}
We say that $\la S,\leq_S\ra$ is a \emph{directed preorder} if it is a preorder (i.e.\ $\leq_S$ is a reflexive and transitive relation on $S$) such that 
\[\forall\, x, y\in S\ \exists\, z\in S\colon x\leq_S z\text{ and }y\leq_S z.\] 
A directed preorder $\la S,\leq_S\ra$ is seen as the relational system $S=\la S, S,\leq_S\ra$, and their associated cardinal characteristics are denoted by $\bfrak(S)$ and $\dfrak(S)$. The cardinal $\dfrak(S)$ is actually the \emph{cofinality of $S$}, typically denoted by $\cof(S)$ or $\cf(S)$.
\end{definition}

Recall that, whenever $S$ is a directed preorder without maximum element, $\bfrak(S)$ is infinite and regular, and $\bfrak(S)\leq\cf(\dfrak(S))\leq\dfrak(S)\leq|S|$. Even more, if $L$ is a linear order without maximum then $\bfrak(L)=\dfrak(L)=\cf(L)$.

The cardinal characteristics associated with an ideal can be characterized by relational systems as well.

\begin{example}\label{exm:Iwf}
For $\Iwf\subseteq\pts(X)$, define the relational systems: 
\begin{enumerate}[label=(\arabic*)]
    \item $\Iwf:=\la\Iwf,\Iwf,\subseteq\ra$, which is a directed preorder when $\Iwf$ is closed under unions (e.g.\ an ideal).
    
    \item $\Cbf_\Iwf:=\la X,\Iwf,\in\ra$.
\end{enumerate}
It is well-known that, whenever $\Iwf$ is an ideal on $X$ containing $[X]^{<\aleph_0}$,
\begin{multicols}{2}
\begin{enumerate}[label= \rm (\alph*)]
    \item $\bfrak(\Iwf)=\add(\Iwf)$. 
    
    \item $\dfrak(\Iwf)=\cof(\Iwf)$.
    
    \item $\bfrak(\Cbf_\Iwf)=\non(\Iwf)$.

    \item $\dfrak(\Cbf_\Iwf)=\cov(\Iwf)$. 
\end{enumerate}
\end{multicols}
\end{example}

    
    
    

\begin{example}\label{ex:ideal<theta}
Let $\theta$ be an infinite cardinal and $X$ a set of size ${\geq}\theta$. Then $[X]^{<\theta}$ is an ideal. 
Its additivity and uniformity numbers are easy to calculate:
\[\add([X]^{<\theta})=\cf(\theta) \text{ and } \non([X]^{<\theta})=\theta.\]

For the covering number, we obtain
    \[\cov([X]^{<\theta})=\left\{
    \begin{array}{ll}
        |X| & \text{if $|X|>\theta$,} \\
        \cf(\theta) & \text{if $|X|=\theta$.}
    \end{array}\right.\]
Therefore $\cov([X]^{<\theta})=|X|$ whenever $\theta$ is regular.

The cofinality number is more interesting. We have the equation
\[|X|^{<\theta}=\max\{2^{<\theta},\cof([X]^{<\theta})\},\]
so $\cof([X]^{<\theta})=|X|^{<\theta}$ whenever $|X|\geq 2^{<\theta}$ (because $\cof([X]^{<\theta}) \geq |X|$). 

More generally, when $\theta$ is regular, $\cof([\theta]^{<\theta})=\theta$, otherwise $\theta<\cof([\theta]^{<\theta})\leq\theta^{<\theta}$; if $\kappa\geq\theta$ is a cardinal, then $\cof([\kappa^+]^{<\theta})=\kappa^+\cdot\cof([\kappa]^{<\theta})$; and, whenever $\lambda>\theta$ is a limit cardinal, $\cof([\lambda]^{<\theta})=\sup\{\cof([\mu]^{<\theta}):\, \theta\leq\mu<\lambda\}$ if $\cf(\lambda)\geq\theta$, otherwise $\lambda<\cof([\lambda]^{<\theta})\leq\lambda^{<\theta}$.

Under Shelah's Strong Hypothesis,\footnote{The failure of this hypothesis requires large cardinals.} it follows that
\[\cof([X]^{<\theta})=\left\{
    \begin{array}{ll}
        |X| & \text{if $\cf(|X|)\geq\theta$,} \\
        |X|^+ & \text{otherwise.}
    \end{array}\right.\]
\end{example}

More examples of relational systems can be obtained through the products and quotients of relational systems. These notions are useful to deal with our main results about the cofinality of $\SNwf$. 

\begin{definition}\label{prodR}
Let $\overline{\Rbf}:=\la \Rbf_i:\, i\in w\ra$ where each $\Rbf_i=\la X_i, Y_i, \sqsubset_i\ra$ is a relational system. Define the \emph{product relational system}
\[\prod\overline{\Rbf}:= \prod_{i\in w} \Rbf_i = \Big\la\prod_{i\in w}X_i,\prod_{i\in w}Y_i,\sqsubset\Big\ra,\]
where $x \sqsubset y$ iff $x_i \sqsubset_i y_i$ for all $i\in w$.

When $I\subseteq\Pwf(w)$ is an ideal, define the \emph{$I$-product}
\[\prod\overline{\Rbf}\Big/I := \prod_{i\in w} \Rbf_i \Big/I = \Big\la\prod_{i\in w}X_i,\prod_{i\in w}Y_i,\sqsubset^I\Big\ra,\]
where $x \sqsubset^I y$ iff $\set{i\in w}{x_i \nsqsubset_i y_i}\in I$. When $\Rbf_i=\Rbf$ for all $i\in w$, we denote $\Rbf^w=\prod\overline{\Rbf}$ and $\Rbf^w/I:=\prod\overline{\Rbf}/I$.
\end{definition}

In the previous definition there is no need to assume $[w]^{<\aleph_0}\subseteq I$, e.g.\ $I=\{\emptyset\}$ is allowed, in which case $\prod\overline\Rbf/I = \prod\overline\Rbf$.

The following list of examples is relevant to the main results of this paper. 

\begin{example}\label{exm:dirpow}
Let $w$ be a non-empty set, and let $\la S,\leq_S \ra$ be a directed preorder.
\begin{enumerate}[label= \rm (\arabic*)]
\item According to \autoref{prodR}, $S^w=\la S^w,S^w,\leq \ra$ where 
\[x\leq y \text{ iff }\forall\, i\in w\colon x(i)\leq_S y(i).\]
This is a directed preorder. Denote $\bfrak_{S}^w(\leq):=\bfrak(S^w)$ and $\dfrak_{S}^w(\leq):=\dfrak(S^w)$. It is known from~\cite[Lem.~2.9]{cardona} that $\bfrak^w_S(\leq) = \bfrak(S)$ (see also \autoref{FacprodRS}~\ref{FacprodRSb}).  

\item Assume that $\delta$ is a limit ordinal. Let $[\delta]^\bd$ be the ideal of bounded subsets of $\delta$. As in \autoref{prodR}, $S^{\delta}/[\delta]^\bd =\la S^\delta,S^\delta,\leq^* \ra$ where 
\[x\leq^* y \text{ iff } \exists\, \beta<\delta\ \forall\, \alpha\in[\beta,\delta)\colon x(\alpha)\leq_S y(\alpha).\] 
Set $\bfrak_{S}^\delta:=\bfrak(S^{\delta}/[\delta]^\bd)$ and $\dfrak_{S}^\delta:=\dfrak(S^{\delta}/[\delta]^\bd)$. It will be clear from \autoref{lem:Ipow} that $\dfrak^\delta_S = \dfrak^\delta_S(\leq)$.

When $\lambda$ is an infinite cardinal and $\lambda=S$ (with its usual order) denote $\bfrak_{\lambda}^\lambda$ and $\dfrak_{\lambda}^\lambda$ by $\bfrak_{\lambda}$ and $\dfrak_{\lambda}$, respectively. These are the well-known \textit{unbounding number} and \textit{dominating number of $\lambda^\lambda$}, respectively.

\item The case $S=\lambda=\omega$ in the previous example gives us the cardinal characteristics $\bfrak:=\bfrak_\omega$ and $\dfrak:=\dfrak_\omega$, the well-known \emph{bounding number} and \emph{dominating number}, respectively.
\end{enumerate}

\end{example}


Inequalities between cardinal characteristics associated with relational systems can be determined by the dual of a relational system and also via Tukey connections. 
%
Fix a relational system $\Rbf=\la X,Y,\sqsubset\ra$. The \emph{dual of $\Rbf$} is the relational system $\Rbf^\perp:=\la Y,X,\sqsubset^\perp\ra$ where $y \sqsubset^\perp x$ iff $\neg(x \sqsubset y)$. 
%
Note that $\dfrak(\Rbf^\perp)=\bfrak(\Rbf)$ and $\bfrak(\Rbf^\perp)=\dfrak(\Rbf)$.

Let 
$\Rbf'=\la X',Y',\sqsubset'\ra$ be another relational system. We say that $(\Psi_-,\Psi_+)\colon\Rbf\to\Rbf'$ is a \emph{Tukey connection from $\Rbf$ into $\Rbf'$} if 
 $\Psi_-\colon X\to X'$ and $\Psi_+\colon Y'\to Y$ are functions such that  \[\forall\, x\in X\ \forall\, y'\in Y'\colon \Psi_-(x) \sqsubset' y' \Rightarrow x \sqsubset \Psi_+(y').\]
The \emph{Tukey order} between relational systems is defined by
$\Rbf\leqT\Rbf'$ iff there is a Tukey connection from $\Rbf$ into $\Rbf'$. \emph{Tukey equivalence} is defined by $\Rbf\eqT\Rbf'$ iff $\Rbf\leqT\Rbf'$ and $\Rbf'\leqT\Rbf$. 


Recall that $\Rbf\leqT\Rbf'$ implies $(\Rbf')^\perp\leqT\Rbf^\perp$, $\dfrak(\Rbf)\leq\dfrak(\Rbf')$ and $\bfrak(\Rbf')\leq\bfrak(\Rbf)$. Hence,
$\Rbf\eqT\Rbf'$ implies $\dfrak(\Rbf)=\dfrak(\Rbf')$ and $\bfrak(\Rbf')=\bfrak(\Rbf)$.

\begin{example}\label{ex:tukeysmall}
\ 
\begin{enumerate}[label=\rm (\arabic*)]
    \item\label{it:diag} For any ideal $\Iwf$ on $X$, $\Cbf_\Iwf\leqT\Iwf$ and $\Cbf^\perp_\Iwf\leqT\Iwf$. These determine some of the inequalities in \autoref{diag:idealI}.
    
    \item\label{it:decC_I} If $\theta'\leq\theta$ are infinite cardinals and $\theta\leq|X|\leq|X'|$, then $\Cbf_{[X]^{<\theta}}\leqT \Cbf_{[X']^{<\theta'}}$ and $[X]^{<\theta}\leqT[X']^{<\theta}$.
    
    \item For any cardinal $\mu$, $\Cbf_{[\mu]^{<\mu}}\leqT\mu \leqT \Cbf^\perp_{[\mu]^{<\mu}} \leqT[\mu]^{<\mu}$. In the case when $\mu$ is regular, $[\mu]^{<\mu}\leqT  \Cbf_{[\mu]^{<\mu}}$, so $\add([\mu]^{<\mu})=\cof([\mu]^{<\mu})=\mu$.
    
    \item\label{it:dirS} Let $S$ be a directed preorder without maximum. Then $S^\perp\leqT S$ and $S\leqT\Cbf_{[\dfrak(S)]^{<\bfrak(S)}}$.
    
    \item If $X$ is a perfect Polish space, then $\Mwf(X)\eqT \Mwf(\R)$ and $\Cbf_{\Mwf(X)}\eqT \Cbf_{\Mwf(\R)}$, where $\Mwf(X)$ denotes the ideal of meager subsets of $X$. Therefore, the cardinal characteristics associated with the meager ideal are independent of the perfect Polish space used to calculate it (see~\cite[Ex.~8.32 \&~Thm.~15.10]{Ke2}).
    
    \item Let $X$ be a Polish space. Denote by $\Bwf(X)$ the $\sigma$-algebra of Borel subsets of $X$, and assume that $\mu\colon\Bwf(X)\to [0,\infty]$ is a $\sigma$-finite measure such that $\mu(X)>0$ and every singleton has measure zero. Then $\Nwf(\mu)\eqT \Nwf(\Lb)$ and $\Cbf_{\Nwf(\mu)}\eqT \Cbf_{\Nwf(\Lb)}$ where $\Nwf(\mu)$ denotes the ideal generated by the $\mu$-measure zero sets and $\Lb$ is the Lebesgue measure on $\R$. Therefore, the cardinal characteristics associated with both measure zero ideals are the same (see~\cite[Thm.~17.41]{Ke2}).
\end{enumerate}
\end{example}

We prove the following curious result.

\begin{fact}\label{goodcof}
Let $\theta$ be an infinite cardinal and $X$ a set of size ${\geq}\theta$. 
Then $\Cbf_{[X]^{<\theta}}\eqT [X]^{<\theta}$ iff $\cof([X]^{<\theta})=|X|$ and $\theta$ is regular.
\end{fact}
\begin{proof}
    Assume that $\Cbf_{[X]^{<\theta}}\eqT [X]^{<\theta}$, so $\add([X]^{<\theta}) = \non([X]^{<\theta})$ and $\cov([X]^{<\theta}) = \cof([X]^{<\theta})$. The first inequality implies $\cf(\theta)=\theta$, and the second that $|X| = \cof([X]^{<\theta})$.

    For the converse, assume that $\theta$ is regular and let $\Cwf\subseteq [X]^{<\theta}$ be a cofinal family of size $|X|$. By \autoref{ex:tukeysmall}~\ref{it:diag}, it is enough to show that $\la \Cwf,\subseteq\ra \leqT \Cbf_{[\Cwf]^{<\theta}}$. The identity map on $\Cwf$ and the map $G\colon [\Cwf]^{<\theta} \to \Cwf$ such that $\bigcup H\subseteq G(H)$ is the required Tukey connection.
\end{proof}





We now look at a few facts about the products of relational systems and quotients. These generalize some results about $S^\lambda$ and $S^\lambda/[\lambda]^{\bd}$ presented in \cite[Sec.~2]{cardona}.

\begin{fact}[Generalization of~{\cite[Lemma~2.9]{cardona}}]\label{FacprodRS}
Let $\overline{\Rbf}:=\la \Rbf_i:\, i\in w\ra$ where each $\Rbf_i=\la X_i, Y_i, \sqsubset_i\ra$ is a relational system. Then:
\begin{enumerate}[label=\rm (\alph*)]
    \item\label{FacprodRSa} $\Rbf_i\leqT\prod \overline{\Rbf}$ for any $i\in w$.
    \item\label{FacprodRSb} $\bfrak(\prod\overline{\Rbf})= \min\set{\bfrak(\Rbf_i)}{i\in w}$.
    \item\label{FacprodRSc} $\sup\set{\dfrak(\Rbf_i)}{i\in w}\leq\dfrak(\prod\overline{\Rbf})\leq\prod_{i\in w}\dfrak(\Rbf_i)$.
    \item\label{FacprodRSd} If $v\subseteq w$ then $\prod_{i\in v}\Rbf_i\leqT \prod\overline{\Rbf}$.
    \item\label{FacprodRSe} If $I\subseteq J$ are ideals on $w$ then $\prod \overline{\Rbf}/J \leqT \prod \overline{\Rbf}/I \leqT \prod\overline{\Rbf}$. In particular, $\bfrak(\prod \overline{\Rbf})\leq \bfrak(\prod \overline{\Rbf}/I)\leq \bfrak(\prod \overline{\Rbf}/J)$ and $\dfrak(\prod \overline{\Rbf}/J)\leq \dfrak(\prod \overline{\Rbf}/I)\leq \dfrak(\prod \overline{\Rbf})$.
    \item\label{FacprodRSf} If all $\Rbf_i$ are the same $\Rbf$ and $|w|<\bfrak(\Rbf)$ then $\dfrak(\Rbf^w)=\dfrak(\Rbf)$.\footnote{When $\Rbf=S$ is a directed preorder, we even have $S^w\eqT S$, which is proved in~\cite[Lemma~2.9]{cardona}.}
\end{enumerate}
\end{fact}

\begin{lemma}[Generalization of~{\cite[Lemma~2.10]{cardona}}]\label{larged/I}
 Let $\overline{\Rbf}$ be as in \autoref{FacprodRS} and let $I$ be an ideal on $w$. Let $\chi_I$ be the smallest cardinal number $\lambda$ such that there does not exist a pairwise disjoint family of size $\lambda$ formed by members of $\Pwf(w)\menos I$.  If $\bfrak(\Rbf_i)$ exists for any $i\in w$ then $\dfrak(\prod \overline{\Rbf}/I)\geq\chi_I$.
\end{lemma}
\begin{proof}
Suppose that $\theta<\chi_I$ and $E=\set{y_\xi}{\xi<\theta}\subseteq\prod_{i\in w}Y_i$. We show that $E$ is not $\prod \overline{\Rbf}/I$-dominating. It suffices to prove that there is an $x\in\prod_{i\in w}X_i$ such that $x\not\sqsubset^I y_\xi$ for all $\xi<\theta$. Since $\theta<\chi_I$, choose a pairwise disjoint family $\set{A_\xi}{\xi<\theta}\subseteq\Pwf(w)\smallsetminus I$. 

For $\xi<\theta$ we know that $\set{y_{\xi,i}}{i\in A_\xi}\subseteq Y_i$. Then, for each $i\in A_\xi$, we can find some $x'_i\in X_i$ such that $x'_i\not\sqsubset_i y_{\xi,i}$ (because $\bfrak(\Rbf_i)$ exists) and, for each $i\in w\smallsetminus \bigcup_{\xi<\theta} A_\xi$, choose some $x''_i\in X_i$ (notice that $X_i\neq\emptyset$). This allows us to define $x\in\prod_{i\in w}X_i$ by 
\[ x_i
   := 
     \begin{cases}
     x'_i & \textrm{if $i\in \bigcup_{\xi<\theta}A_\xi$,}\\
     x''_i   & \textrm{if $i\in w\smallsetminus \bigcup_{\xi<\theta}A_\xi$.}\\
     \end{cases}
\] 
It is clear that, for each $\xi<\theta$, $\set{i\in w}{x_i \nsqsubset_i y_{\xi,i}}\not\in I$ (because it contains $A_\xi$). This finishes the proof.
\end{proof}

The result~\cite[Lemma~2.10]{cardona} states that $\dfrak(S^\lambda/[\lambda]^\bd)>\lambda$ whenever $\lambda$ is an infinite cardinal and $S$ is a directed preorder without maximum. This follows from \autoref{larged/I} because $\chi_{[\lambda]^\bd}=\lambda^+$. More generally, since by \autoref{lem:scf} $\chi_{[\delta]^\bd}=\scf(\delta)^+$ for any limit ordinal $\delta$, we obtain:

\begin{corollary}\label{cord/I}
    Let $\delta$ be a limit ordinal and let $\overline{\Rbf}=\la\Rbf_\alpha:\, \alpha<\delta\ra$ be a sequence of relational systems such that $\bfrak(\Rbf_\alpha)$ exists for all $\alpha<\delta$. Then $\dfrak(\prod\overline{\Rbf}/[\delta]^\bd)>\scf(\delta)$.
\end{corollary}

We now illustrate the deep connection we have between $\dfrak(\prod\overline{\Rbf})$ and $\dfrak(\prod\overline{\Rbf}/I)$.

\begin{theorem}[Generalization of~{\cite[Thm.~2.11]{cardona}}]\label{bdd/I}
 Let $\overline{\Rbf}$ be as in \autoref{FacprodRS} and let $I$ be an ideal on $w$. Then
 \begin{align*}
     \max\Big\{\dfrak\big(\prod\overline{\Rbf}/I\big),\sup_{u\in I}\dfrak\Big(\prod_{i\in u}\Rbf_i\Big)\Big\} & \leq \dfrak(\prod\overline{\Rbf})\\ & \leq \dfrak\big(\prod\overline{\Rbf}/I\big) \cdot \cof(I) \cdot \sup_{u\in I}\dfrak\Big(\prod_{i\in u}\Rbf_i\Big).
 \end{align*}
\end{theorem}

In particular, under the hypothesis of \autoref{cord/I}, equality holds because $\cof([\delta]^\bd)=\cf(\delta)\leq\scf(\delta)$.

\begin{proof}
The inequality \[\max\Big\{\dfrak\big(\prod\overline{\Rbf}/I\big),\sup_{u\in I}\dfrak\Big(\prod_{i\in u}\Rbf_i\Big)\Big\} \leq\dfrak(\prod\overline{\Rbf})\]
follows from~\autoref{FacprodRS}~\ref{FacprodRSd},~\ref{FacprodRSe}. 
Now we prove the inequality \[\dfrak(\prod\overline{\Rbf})\leq\dfrak\big(\prod\overline{\Rbf}/I\big)\cdot\cof(I)\cdot\sup_{u\in I}\dfrak\Big(\prod_{i\in u}\Rbf_i\Big).\]
For $u\in I$ let $D_u$ be a witness of $\dfrak\Big(\prod_{i\in u}\Rbf_i\Big)$, let $J\subseteq I$ be a witness of $\cof(I)$, and let $D\subseteq\prod_{i\in w}Y_i$ be a witness of $\dfrak\big(\prod\overline{\Rbf}/I\big)$. For $u\in J$, $y\in D_u$ and $z\in D$, define the function $f_{y,z}^u\in\prod_{i\in w}Y_i$ by 
\[ f_{y,z}^u(i)
   := 
     \begin{cases}
     y_i & \textrm{if $i\in u$,}\\
      z_i  & \textrm{if $i\in w\smallsetminus u$.}\\
     \end{cases}
\] 
Notice that $|\set{f_{y,z}^u}{z\in D,\ \exists\, u\in J\colon y\in D_u}|\leq\dfrak\big(\prod\overline{\Rbf}/I\big)\cdot\cof(I)\cdot\sup_{u\in I}\dfrak\Big(\prod_{i\in u}\Rbf_i\Big)$. Hence, to complete the argument, we shall prove that this family is $\prod\overline{\Rbf}$-dominating. Fix $x\in\prod_{i\in w}X_i$. Next, find $z\in D$ such that $u:=\set{i\in w}{x_i \nsqsubset_i z_i}\in I$.  
Now choose $u'\in J$ such that $u'\supseteq u$. Then, we can find $y\in D_{u'}$ such that $x{\upharpoonright}u'\sqsubset y$. It is easy to see that $x\sqsubset f_{y,z}^{u'}$.
\end{proof}


We can say more in the case of $I$-powers when $I\subseteq [\lambda]^{<\lambda}$ and $\lambda$ is an infinite cardinal.

\begin{theorem}\label{lem:Ipow}
    Let $\Rbf=\la X,Y,\sqsubset\ra$ be a relational system, $\lambda$ an infinite cardinal, and let $I\subseteq [\lambda]^{<\lambda}$ be an ideal. Then
    $\dfrak(\Rbf^\lambda/I)= \dfrak(\Rbf^\lambda)$.
\end{theorem}
\begin{proof}
    The inequality $\leq$ follows by \autoref{FacprodRS}~\ref{FacprodRSe}. To show the converse inequality, thanks to \autoref{FacprodRS}~\ref{FacprodRSe}, it is enough to assume that $I=[\lambda]^{<\lambda}$. The case when $\dfrak(\Rbf)=1$ is trivial, so also assume that $\dfrak(\Rbf)>1$, which is equivalent to saying that $\bfrak(\Rbf)$ exists. Then, by \autoref{larged/I}, $\nu:=\dfrak(\Rbf^\lambda/I)\geq\lambda^+$. Choose a dominating family $\{y_\alpha:\, \alpha<\nu\}$ on $\Rbf^\lambda/I$ and some bijection $f\colon \lambda\times \lambda \to \lambda$. For each $\alpha<\nu$ and $\beta<\lambda$, define $y^\beta_\alpha\in Y^\lambda$ by $y^\beta_\alpha(\xi):= y_\alpha(f(\beta,\xi))$. Since $\nu>\lambda$, it remains to show that $\set{y^\beta_\alpha}{\alpha<\nu,\ \beta<\lambda}$ is $\Rbf^\lambda$-dominating. Let $x\in X^\lambda$ and define $x'\in X^\lambda$ by $x'(\eta):=x(\xi)$ when $f(\beta,\xi)=\eta$ for some $\beta<\lambda$. We then obtain that $x'\sqsubset^I y_\alpha$ for some $\alpha<\nu$, that is, $v:=\set{\eta<\lambda}{x'(\eta)\nsqsubset y_\alpha(\eta)}$ has size ${<}\lambda$. Then, there is some $\beta<\lambda$ such that $f(\beta,\xi)\notin v$ for all $\xi<\lambda$, hence $x(\xi)=x'(f(\beta,\xi))\sqsubset y_\alpha(f(\beta,\xi))=y^\beta_\alpha(\xi)$. Therefore, $x$ is $\Rbf^\lambda$-dominated by $y^\beta_\alpha$.
\end{proof}

The previous result improves~\cite[Prop.~12]{brehig}: the cited result states that, whenever $\kappa\leq\lambda$ are cardinals, $\dfrak(\kappa^\lambda/[\lambda]^\bd)=\dfrak(\kappa^\lambda)$ and $\dfrak(\kappa^\lambda)=\dfrak(\kappa^\lambda/[\lambda]^{<\kappa})$, the latter whenever $\cf(\lambda)\geq\kappa$. However, the assumption $\cf(\lambda)\geq\kappa$ is not required anymore by \autoref{lem:Ipow} and, even more, $\dfrak(\kappa^\lambda/[\lambda]^{<\lambda})=\dfrak(\kappa^\lambda)$. In addition, \autoref{lem:Ipow} improves \cite[Thm.~2.11]{cardona}, which states that $\dfrak(S^\lambda)=\sup(\{\dfrak(S^\lambda/[\lambda]^\bd)\}\cup\set{\dfrak(S^\alpha)}{\alpha<\lambda})$ whenever $S$ is a directed preorder.

The product of relational systems can be used to produce powers of ideals as follows.

\begin{definition}\label{def:idpow}
Given an ideal $\Iwf$ on $X$ and a set $w$, define $\Iwf^{(w)}$ as the ideal on $X^w$ generated by the sets of the form $\prod_{i\in w}A_i$ with $\Seq{A_i}{i\in w}\in\Iwf^{w}$. Denote $\add(\Iwf^w):=\bfrak(\Iwf^{(w)})$, $\cof(\Iwf^w):=\dfrak(\Iwf^{(w)})$, $\non(\Iwf^w):=\bfrak(\Cbf_{\Iwf^{(w)}})$ and $\cov(\Iwf^w):=\dfrak(\Cbf_{\Iwf^{(w)}})$.
\end{definition}

\begin{fact}\label{fct:idpow}
Let $w$ be a set and let $\Iwf$ be an ideal on $X$. Then:
\begin{enumerate}[label=\rm (\alph*)]
    \item $\Iwf^w\eqT\Iwf^{(w)}$. 
    
    \item $\Cbf_\Iwf^w\eqT\Cbf_{\Iwf^{(w)}}$.
    
    \item $\add(\Iwf^w)=\add(\Iwf)$ and $\non(\Iwf^w)=\non(\Iwf)$.
    
    \item $\cov(\Iwf)\leq\cov(\Iwf^w)\leq \cov(\Iwf)^{|w|}$ and $\cof(\Iwf)\leq\cof(\Iwf^w)\leq \cof(\Iwf)^{|w|}$.
\end{enumerate}
\end{fact}
\begin{proof}
 The last two properties follow from \autoref{FacprodRS}.
\end{proof}

\section{Bounds for the cofinality
}\label{BounSN}

The purpose of this section is to modify, in concise form, various notions of directed systems and dominating (directed) systems from~\cite{cardona} that are needed to prove~\autoref{new_upperbI}--\ref{lowerSN:I}. These notions are one of the main tools we use to study the principle $\DS(\delta)$ and treat the cofinality of $\SNwf$. 



We first recall the central notions from~\cite{cardona} (with slight variations).

\begin{definition}[{\cite[Def.~3.5 and~3.6]{cardona}}]\label{def:I_fdirsystS} Let $S$ be a directed  preorder. 
Given $f\in\omega^{\uparrow\omega}$, we say that a family $A^f=\Seq{ A_{i}^f}{i\in S}$ of subsets of $2^\omega$ is an \textit{$\Iwf_f$-directed system on $S$} 
if it satisfies:\footnote{The original reference also demanded that every $A^f_i$ is $G_\delta$-dense, but this is not necessary.}
\begin{enumerate}[label=\rm (D\arabic*)]
    \item $\forall\, i\in S\colon A_{i}^f\in\Iwf_{f}$,
    \item\label{it:D3} $\forall\, i, j\in S \colon i\leq_S j\imp A_{i}^f\subseteq A_{j}^f$, and 
    \item $\Seq{A_{i}^{f}}{i\in S}$ is cofinal in $\Iwf_{f}$.
    \end{enumerate}
If $\delta$ is an ordinal and $\bar f=\Seq{f_\alpha}{\alpha<\delta}$ forms a dominating family in $\omega^{\uparrow\omega}$ (so $\dfrak\leq\delta$), we say that $\Seq{\bar A^{f_\alpha}}{ \alpha<\delta}$ is an \textit{$\bar f$-dominating directed system\footnote{The original definition is called ``$\lambda$-dominating system", but we emphasize in this paper the dominating family instead.} on 
$S$}
if each $\bar A^{f_\alpha}$ is an $\Iwf_{f_\alpha}$-directed system on 
$S$ and 
\begin{enumerate}[resume*]
    \item\label{it:D5} 
    For any $\alpha<\delta$ and $z\in S^\alpha$, $\bigcap_{\xi<\alpha}A_{z(\xi)}^{f_\xi} \notin \Iwf_{f_\alpha}$.\footnote{In the original definition, it is assumed that $S$ has a minimum element $i_0$ and this condition was stated only for $z(\xi)=i_0$ for all $\xi<\alpha$ (which is equivalent by~\ref{it:D3}).}
    \end{enumerate}
\end{definition}

We list below the results from~\cite{cardona} about directed systems and the cofinality of $\SNwf$. For these, fix a directed preorder $S$.

\begin{lemma}[{\cite[Lemma~3.7]{cardona}}]\label{C3.7}
Assume $\cov(\Mwf)=\dfrak$, $D\subseteq\baireincr$ is dominating and, for each $f\in D$, there is some $\Iwf_f$-directed system on $S$. 
Then there is some $\bar f=\la f_\alpha:\, \alpha<\dfrak\ra$, forming a dominating family contained in $D$, and there is some  $\bar f$-dominating directed system on $S$. 
\end{lemma}

\begin{theorem}[{\cite[Thm~3.8]{cardona}}]\label{C3.8}
Assume that there is some $\bar f$-dominating directed system on 
$S$ with $\bar f$ of length $\delta$. Then $\SNwf\leqT S^\delta$. In particular, $\bfrak(S)\leq\add(\SNwf)$ and $\cof(\SNwf)\leq\dfrak^\delta_S$.
\end{theorem}

The hypothesis of the previous theorem can be considerably weakened, even to conclude the stronger $\SNwf\leqT S^\dfrak$. To show this, and to prove our main results as well, we simplify \autoref{def:I_fdirsystS} without using directed preorders. This also leads us to generalize~\autoref{C3.7}.



\begin{definition}\label{def:I_fsystS} Let $I$ be a set, 
$\delta$ an ordinal, and let $\bar f=\Seq{f_\alpha}{\alpha<\delta}$ form a dominating family in $\baireincr$. We say that $\Seq{\bar A^{f_\alpha}}{ \alpha<\lambda}$ is an $\bar f$-\textit{dominating system on $I$}
if, for each $\alpha<\delta$: 
\begin{enumerate}[label= \rm (I\arabic*)]
    \item $\bar A^{f_\alpha}=\la A^{f_\alpha}_i:\, i\in I\ra$ forms a cofinal family on $\Iwf_{f_\alpha}$, 
      and 
    \item\label{it:I5} 
    for all $\alpha<\delta$ and $z\in I^\alpha$, $\bigcap_{\xi<\alpha} A^{f_\xi}_{z(\xi)}\notin \Iwf_{f_\alpha}$.
\end{enumerate}

We frequently denote $\bar A^{\alpha}:=\bar A^{f_\alpha}$ and  $A^\alpha_i:=A^{f_\alpha}_i$ for all $\alpha<\lambda$ and $i\in I$.
\end{definition}
 
\begin{fact}\label{rem:Ifsys}
Let $I$ be a set. 
Then:
\begin{enumerate}[label = \rm (\alph*)]
    \item\label{it:blcof} For any $f\in\baireincr$, there exists a cofinal family $\la A^f_i:\, i\in I\ra$ on $\Iwf_f$ 
    iff $\cof(\Iwf_f)\leq|I|$.
    \item\label{it:exdirsys} For all infinite cardinals $\theta\leq\add(\Iwf_f)$ and $\lambda\geq\cof(\Iwf_f)$, 
    there exists an $\Iwf_f$-directed system on $[\lambda]^{<\theta}$. 
\end{enumerate}
\end{fact}
\begin{proof}
We prove~\ref{it:exdirsys}.
Fix some cofinal family $\{C_\alpha:\, \alpha\in\lambda\}$ on $\Iwf_f$. For each $s\in[\lambda]^{<\theta}$ define $A^f_s:=\bigcup_{\alpha\in s}C_\alpha$.
\end{proof}


The following is a very easy fact about Yorioka ideals in terms of Tukey connections, which implies the inequality $\cof(\SNwf)\leq 2^\dfrak$ (already known from~\cite{Ser}). 

\begin{theorem}\label{newupperb}
For any dominating family $D\subseteq\omega^{\uparrow\omega}$, $\SNwf\leqT\prod_{f\in D}\Iwf_f$ . In particular, $\minadd\leq\add(\SNwf)$, 
$\cof(\SNwf)\leq\dfrak\left(\prod_{f\in D}\Iwf_f\right)$ 
and $\cof(\SNwf)\leq 2^\dfrak$.
\end{theorem}
\begin{proof}
 Define the functions $\Psi_-:\SNwf\to\prod_{f\in D}\Iwf_f$ and $\Psi_+:\prod_{f\in D}\Iwf_f\to\SNwf$ where $\Psi_-(A)$ is the constant sequence with value $A$, and $\Psi_+(\la B_f:\, f\in D\ra):=\bigcap_{f\in D}B_f$, which is in $\SNwf$ because $\SNwf=\bigcap_{f\in D}\Iwf_f$. These give the desired Tukey connection.



By~\autoref{FacprodRS}, $\minadd=\bfrak(\prod_{f\in D}\Iwf_f)\leq\add(\SNwf)$ and $\cof(\SNwf)\leq\dfrak(\prod_{f\in D}\Iwf_f)$. If we choose a dominating family $D$ of size $\dfrak$, we obtain $\dfrak(\prod_{f\in D}\Iwf_f)\leq\prod_{f\in D}\cof(\Iwf_f)=2^\dfrak$.
\end{proof}

As a consequence, \autoref{C3.8} is simplified as follows.

\begin{theorem}\label{simpleC3.8}
 Let $S$ be a directed preorder.
 Assume that there is a dominating family $D\subseteq\baireincr$ such that, for each $f\in D$, there is some 
 $\Iwf_f$-directed system on $S$. 
 Then $\SNwf\leqT S^\dfrak$. In particular, $\bfrak(S)\leq\minadd\leq\add(\SNwf)$ and $\cof(\SNwf)\leq \dfrak(S^\dfrak)$.
\end{theorem}
\begin{proof}
Choose $D'\subseteq D$ dominating of size $\dfrak$.
Note that the existence of a $\Iwf_f$-directed system on $S$ 
implies that $\Iwf_f\leqT S$. Therefore, $\SNwf\leqT \prod_{f\in D'}\Iwf_f\leqT S^{D'}\eqT S^\dfrak$.
\end{proof}

\autoref{new_upperbI} is a consequence of \autoref{new_upperb}, where we improve the upper bound of $\cof(\SNwf)$ in ZFC:

\begin{theorem}\label{new_upperb}
$\SNwf\leqT\Cbf_{[\supcof]^{<\minadd}}^\dfrak$. In particular, 
\[\cof(\SNwf)\leq\cov\left(\left([\supcof]^{<\minadd}\right)^\dfrak\right).\]
\end{theorem}
\begin{proof}
Choose a dominating family $D\subseteq\baireincr$ of size $\dfrak$. As a consequence of \autoref{newupperb}, 
and \autoref{ex:tukeysmall}~\ref{it:decC_I},~\ref{it:dirS}, we obtain
\[\SNwf \leqT\prod_{f\in D}\Iwf_f \leqT\prod_{f\in D}\Cbf_{[\cof(\Iwf_f)]^{<\add(\Iwf_f)}} \leqT\prod_{f\in D}\Cbf_{[\supcof]^{<\minadd}} =\Cbf_{[\supcof]^{<\minadd}}^\dfrak.\qedhere\]
%
\end{proof}

\begin{remark}
Inequalities in \autoref{newupperb} and~\ref{new_upperb} can be strict. For example, in Laver's model, $\cof(\SNwf)=\aleph_2$, but $\dfrak\left(\prod_{f\in D}\Iwf_f\right)>|D|=\aleph_2$ for any $D\subseteq\baireincr$ dominating (by \autoref{larged/I} applied to $I=\{\emptyset\}$).
\end{remark}

The previous theorems basically describe possible upper bounds of $\cof(\SNwf)$, and for this, we have not used dominating (directed) systems at all. They are actually going to be used to calculate lower bounds of $\cof(\SNwf)$, as it is done in~\cite{cardona}.

\begin{theorem}[{\cite[Thm.~3.12]{cardona}}]\label{thm:C3.12}
  Assume that $\kappa$ and $\lambda$ are cardinals such that $0<\kappa\leq\lambda\leq\non(\SNwf)$ and $\bar f=\la f_\alpha<\alpha<\lambda\ra$ dominating in $\baireincr$. If there is some $\bar f$-dominating directed system on $\kappa\times\lambda$ then
  $\lambda^\lambda\leqT\SNwf$, in particular $\add(\SNwf)\leq\cf(\lambda)$ and $\dfrak_\lambda\leq\cof(\SNwf)$.
\end{theorem}

We prove variations of \autoref{C3.7} and \autoref{thm:C3.12} in terms of the simpler ``dominating systems". Moreover, we show that our results generalize the previous theorem, as well as Yorioka's characterization of $\cof(\SNwf)$.

We look at the existence of dominating systems, which does not depend neither on $I$ nor on the cofinal families.

\begin{lemma}\label{DSchar}
Let $\delta$ be an ordinal and let $I$ be a set of size $\geq\supcof$. Assume that $\bar f=\la f_\alpha:\, \alpha<\delta\ra$ is dominating on $\baireincr$. Then, the following statements are equivalent.
\begin{enumerate}[label =\rm (\roman*)]
    \item\label{it:exDSI} There exists an $\bar f$-dominating system on $I$.    \item\label{it:DSsimple} There is some sequence $\bar A=\la A_\alpha:\, \alpha<\delta\ra$ of sets such that, for each $\alpha<\delta$,
    \begin{multicols}{2}
    \begin{enumerate}[label =\rm (S\arabic*)]
        \item\label{S1} $A_\alpha\in \Iwf_{f_\alpha}$ and
        \item\label{S2} $\bigcap_{\xi<\alpha} A_\xi\notin \Iwf_{f_\alpha}$.
    \end{enumerate}
    \end{multicols}
\end{enumerate}
\end{lemma}
\begin{proof}
It is clear that~\ref{it:exDSI} implies~\ref{it:DSsimple}, so we prove the converse. Assume that $\bar A$ satisfies~\ref{S1} and~\ref{S2}. For each $\alpha<\delta$, since $\supcof\leq|I|$, we can find a cofinal family $\bar A^\alpha =\la A^\alpha_i:\, i\in I\ra$ on $\Iwf_{f_\alpha}$ such that $A_\alpha\subseteq A^\alpha_i$ for all $i\in I$. It is clear that $\la \bar A^\alpha:\, \alpha<\delta\ra$ forms an $\bar f$-dominating system.
\end{proof}

A similar argument to prove \autoref{C3.7} actually works to prove the stronger \autoref{lem:DomSys}. First recall the following result that expresses, in terms of Yorioka ideals, the fact that no perfect set of reals can be of strong measure zero.

\begin{lemma}[{\cite[Lemma~3.7]{Yorioka}}]\label{perfset}
Let $A$ be a perfect subset of $2^\omega$. Then there some $f\in\omega^{\uparrow\omega}$ such that $A\notin\Iwf_f$. 
\end{lemma}

\begin{lemma}\label{lem:DomSys}
 Let $I$ be a set. Assume 
 \begin{enumerate}[label = \rm(\roman*)]
     \item $\cov(\Mwf)=\dfrak$,
     \item $D\subseteq\omega^{\uparrow\omega}$ is a dominating family and,
     \item $\supcof\leq|I|$. 
 \end{enumerate}
    Then there is some $\bar f=\la f_\alpha:\, \alpha<\dfrak\ra$, forming a dominating family contained in $D$, and there is some $\bar f$-dominating system on $I$. 
\end{lemma}
\begin{proof}
It is enough to construct $\bar f$ and $\bar A=\la A_\alpha:\, \alpha<\dfrak\ra$ as in \autoref{DSchar}~\ref{it:DSsimple}, where $A_\alpha = [\sigma_\alpha]_\infty$ for some $\sigma_\alpha\in (2^{<\omega})^\omega$. 
Let $\lambda:=\dfrak$ and let $\Seq{ h_\alpha}{\alpha<\lambda}$ be a dominating family. 
By recursion on $\alpha<\lambda$, we construct $f_\alpha$ and $A_\alpha$, where the latter is dense $G_\delta$. 

Assume that $\Seq{f_\beta}{ \beta<\alpha}$ and $\la A_\beta:\, \beta<\alpha\ra$ have been constructed. We can get a transitive model $M$ of (a large enough fragment of) ZFC such that $|M|<\lambda=\cov(\Mwf)$ and $A_\beta$ is coded in $M$ for any $\beta<\alpha$, i.e., $\sigma_\beta\in M$.

Recall that $M<\cov(\Mwf)$ implies that there is a Cohen real over $M$, hence it adds a perfect set $P$ of Cohen reals over $M$ (see e.g.~\cite[Lemma~3.3.2]{BJ}). Since each $A_\beta$ ($\beta<\alpha$) is a dense $G_\delta$ set coded in $M$, $P\subseteq \bigcap_{\beta<\alpha}A_{\beta}$. On the other hand, there is some $g\in\omega^\omega$ such that $P\notin\Iwf_{g}$ by~\autoref{perfset}.

Choose some $f_\alpha\in D$ such that $h_\alpha\leq f_\alpha$ and $g\leq f_\alpha$. Then $\Iwf_{f_\alpha}\subseteq\Iwf_g$ and  $P\notin\Iwf_{f_\alpha}$. But $P\subseteq \bigcap_{\beta<\alpha}A_{\beta}$, hence $\bigcap_{\beta<\alpha}A_{\beta}\notin\Iwf_{f_\alpha}$. Finally, choose some dense $G_\delta$ $A_\alpha=[\sigma_\alpha]_\infty\in\Iwf_{f_\alpha}$. This finishes the construction.

Clearly, $\bar f:=\Seq{ f_\alpha}{ \alpha<\lambda}$ is a dominating family and $\Seq{ A_\alpha}{\alpha<\lambda}$ is as required.
\end{proof}

Motivated by this result, we introduce a principle expressing the existence of dominating systems, which we use to prove our main theorems about the cofinality of $\SNwf$.

\begin{definition}\label{def:axDS}
  We say that a pair $(\bar f,\bar A)$ is a \emph{$\DS$-pair} if $\bar f=\la f_\alpha:\, \alpha<\delta\ra$ forms a dominating family on $\baireincr$ (for some ordinal $\delta$) and $\bar A= \la A_\alpha:\, \alpha<\delta\ra$ is a sequence of sets satisfying~\ref{S1} and~\ref{S2} of \autoref{DSchar}~\ref{it:DSsimple} for any $\alpha<\delta$. Here, $\delta$ is the \emph{length of the $\DS$-pair $(\bar f,\bar A)$}.
  
  Given an ordinal $\delta$, we define the following statement:
  \begin{enumerate}[label = {$\DS(\delta)$},leftmargin=62pt]
     \item : There is some $\DS$-pair $(\bar f,\bar A)$ of length $\delta$.
  \end{enumerate}
\end{definition}

Thanks to \autoref{DSchar},
we can reformulate \autoref{lem:DomSys} as follows.

\begin{lemma}\label{lem:cov=d}
$\cov(\Mwf)=\dfrak$ implies $\DS(\dfrak)$.
\end{lemma}


We now look at the combinatorics we have from $\DS(\delta)$.

\begin{lemma}\label{lem:ondelta}
Assume $\DS(\delta)$ witnessed by $\bar f=\la f_\alpha:\, \alpha<\delta\ra$.  
Then:
\begin{enumerate}[label=\rm (\alph*)]
    \item\label{it:restrdom} If $c\subseteq \delta$ with order type $\rho$ and $\bar f\frestr c:=\Seq{f_\alpha}{\alpha\in c}$ is dominating, then $\DS(\rho)$ holds.
    \item\label{it:noshortdom} For any $\beta<\delta$, $\bar f\frestr\beta$ is \underline{not} dominating.
    \item\label{it:cofdom} For any $\beta<\delta$, $\bar f\frestr(\delta\menos\beta)$ is dominating and $\DS(\delta-\beta)$ holds.
    \item\label{it:deltalarge} $\delta$ is a limit ordinal and $\dfrak\leq\delta<\cfrak^+$. Moreover, the sequences $\bar f$ and $\la\Iwf_{f_\alpha}:\, \alpha<\delta\ra$ are one-to-one.
    
    \item\label{it:delta'} There is some $\delta'\leq\delta$ such that $|\delta'|=\dfrak$, $\cf(\delta')=\cf(\delta)$ and $\DS(\delta')$ holds.
    \item\label{it:cfdelta} $\bfrak\leq\cf(\delta)\leq \dfrak$.

    \item\label{it:normalform} 
    If $\ir(\delta)=\omega^\varepsilon$ (ordinal exponentiation) then $\bfrak\leq \cf(\varepsilon)\leq \dfrak \leq\varepsilon$, and $\DS(\ir(\delta))$ holds.
\end{enumerate}
\end{lemma}
\begin{proof}
Choose $\bar A=\Seq{A_\alpha}{\alpha<\delta}$ such that $(\bar f,\bar A)$ is a $\DS$-pair.

\ref{it:restrdom} Let $h\colon \rho\to c$ be an increasing bijection. Then $\la f_{h(\xi)}:\, \xi<\rho\ra$ and $C_\xi:= A_{h(\xi)}$ ($\xi<\rho$) forms a $\DS$-pair of length $\rho$.

\ref{it:noshortdom}: Assume that $\bar f\frestr \beta$ is dominating for some $\beta<\delta$. Then $\bigcap_{\alpha<\beta}A_\alpha \in\SNwf\subseteq\Iwf_{f_\beta}$, which contradicts property~\ref{S2} of \autoref{DSchar}.

\ref{it:cofdom}: Since $\bar f\frestr\beta$ is not dominating, there is some $x'\in\omega^\omega$ unbounded over $\bar f\frestr \beta$. For any $x\in\omega^\omega$ there is some $y\in\omega^\omega$ bounding $x$ and $x'$. Then, $y\leq^* f_\alpha$ for some $\alpha<\delta$, and since $f_\alpha$ also dominates $x'$ we get $\alpha\geq\beta$. Hence $\bar f\frestr(\delta\menos\beta)$ is dominating. Therefore, $\DS(\delta-\beta)$ holds by~\ref{it:restrdom}.

\ref{it:deltalarge}: It is clear that $\delta\geq\dfrak$ because $\bar f$ is a dominating family. If $\delta=\gamma+1$ then $f_\gamma$ would be an upper bound of $\omega^\omega$ by~\ref{it:cofdom}, which is impossible. Hence $\delta$ is a limit ordinal.

To conclude the proof, it is enough to show that $\la\Iwf_{f_\alpha}:\, \alpha<\delta\ra$ is a one-to-one sequence. Let $\alpha<\beta<\delta$. Since $A_\alpha\in \Iwf_{f_\alpha}$, we get $\bigcap_{\xi<\beta}A_\xi\in\Iwf_{f_\alpha}$, but this is not in $\Iwf_{f_{\beta}}$. Therefore, $\Iwf_{f_\alpha}\nsubseteq\Iwf_{f_\beta}$.

\ref{it:delta'}: There is some $d\subseteq\delta$ of size $\dfrak$ such that $\bar f\frestr d$ is dominating. Let $\delta'$ be the order type of $d$, i.e. there is an increasing enumeration $d=\{\alpha_\xi:\, \xi<\delta'\}$. By~\ref{it:noshortdom}, $d$ is cofinal in $\delta$, so $\cf(\delta')=\cf(\delta)$. It is clear that $|\delta'|=\dfrak$, and $\DS(\delta')$ follows by~\ref{it:restrdom}.

\ref{it:cfdelta}: $\cf(\delta)\leq \dfrak$ is a consequence of~\ref{it:delta'}.
To show $\cf(\delta)\geq\bfrak$, pick an increasing cofinal sequence $\la \alpha_\zeta:\, \zeta<\cf(\delta) \ra$ in $\delta$. Since $\bar f\frestr\alpha_\zeta$ is not dominating, there is some $x_\zeta\in\omega^\omega$ that cannot be dominated by any $f_\alpha$ with $\alpha<\alpha_\zeta$. Then $\set{x_\zeta}{\zeta<\cf(\delta)}$ is unbounded (because no $f_\alpha$ would dominate an upper bound of such a set), so $\bfrak\leq\cf(\delta)$.

\ref{it:normalform}: By~\ref{it:cofdom}, $\DS(\ir(\delta))$ holds, so $\ir(\delta)=\omega^{\varepsilon}\geq\dfrak$ by~\ref{it:deltalarge}. Since $|\omega^{\varepsilon}|$ is either $1$ or $\max\{\omega,|\varepsilon|\}$, we must have $\varepsilon\geq\dfrak$. On the other hand, $\bfrak\leq \cf(\omega^{\varepsilon})\leq \dfrak$ by~\ref{it:cfdelta}, so $\omega^{\varepsilon}$ cannot have countable cofinality. Hence $\varepsilon$ is a limit ordinal and $\cf(\omega^{\varepsilon})=\cf(\varepsilon)$, so $\bfrak\leq \cf(\varepsilon) \leq \dfrak$.
\end{proof}

\begin{corollary}\label{cor:DSreg}
If $\kappa$ is a regular cardinal and $\DS(\kappa)$ holds, then $\kappa=\dfrak$.
\end{corollary}

\begin{corollary}\label{DSclosecard}
If $\DS(\delta)$ holds and $\cf(\delta)=|\delta|$ then $|\delta|=\dfrak$ and $\DS(\dfrak)$ holds.
\end{corollary}
\begin{proof}
Let $\bar f=\la f_\alpha:\, \alpha<\delta\ra$ be a witness of $\DS(\delta)$, and let $\kappa:=\cf(\delta)=|\delta|$, which is regular. 
Let $h\colon \kappa\to\delta$ be a bijection. Since $\cf(\delta)=\kappa$, we can define by recursion an increasing sequence $\la \alpha_\xi:\, \xi<\kappa\ra\subseteq\delta$ such that $f_{h(\xi)}\leq^* f_{\alpha_\xi}$. Then, $\la f_{\alpha_\xi}:\, \xi<\kappa \ra$ is a dominating family, so $\DS(\kappa)$ holds by \autoref{lem:ondelta}~\ref{it:restrdom}. Therefore, $\kappa=\dfrak$ by \autoref{cor:DSreg}.
\end{proof}

\begin{corollary}\label{DSb=d}
If $\DS(\delta)$ holds for some $\delta$ and $\bfrak=\dfrak$, then $\DS(\dfrak)$ holds.
\end{corollary}
\begin{proof}
Assume $\DS(\delta)$.
By \autoref{lem:ondelta}~\ref{it:delta'}, there is some $\delta'\leq\delta$ such that $|\delta'|=\dfrak$, $\cf(\delta')=\cf(\delta)$ and $\DS(\delta')$ holds. Now, $\bfrak=\dfrak$ implies $\cf(\delta')=\dfrak$ by  \autoref{lem:ondelta}~\ref{it:cfdelta}, so $\DS(\dfrak)$ follows by \autoref{DSclosecard}.
\end{proof}

The statement $\DS(\delta)$ puts some restrictions on how a cofinal family on $\SNwf$ should be. This is reflected in the following technical main result of this paper (\autoref{SNpower:I}), which also determines a general tool to construct uncountable strong measure zero sets.

\begin{theorem}\label{SNpower}
Assume $\DS(\delta)$ witnessed by $\bar f=\la f_\alpha:\, \alpha<\delta\ra$. Let $I$ be a set of size ${\geq}\supcof$, and assume that $\Seq{\bar A^\alpha}{\alpha<\delta}$ forms an $\bar f$-dominating system on $I$. Let $\la J_\alpha:\, \alpha<\delta\ra$ be a sequence of sets and assume, for each $\alpha<\delta$,
\begin{enumerate}[label =\rm (H\arabic*)]
    \item $\{C^\alpha_j:\, j\in J_\alpha\}\subseteq \Iwf_{f_\alpha}$, and
    \item $\sum_{\alpha'<\alpha}|J_{\alpha'}|<\non(\SNwf)$.
\end{enumerate}
Then, there are some function $G\colon \delta\to I$ and some set $K\in\SNwf$ satisfying:
\begin{enumerate}[label=\rm (\Roman*)]
    \item\label{itI} $K\subseteq\bigcap_{\gamma<\delta}A_{G(\gamma)}^{\gamma}\nsubseteq C_{j}^{\alpha}$ for all $\alpha<\delta$ and $j\in J_\alpha$, and
    \item\label{it:II} $|K| = \sum_{\alpha<\delta}|J_\alpha|\leq\non(\SNwf)$.
\end{enumerate}
\end{theorem}
\begin{proof}
By recursion on $\alpha<\delta$, we construct $G$ and $\set{x_{j}^{\alpha}}{ \alpha<\delta,\ j\in J_\alpha}\subseteq 2^\omega$ satisfying
\begin{enumerate}[label= \rm (\roman*)]
     \item\label{constone} $\set{x_j^{\alpha'}}{ \alpha'\leq\alpha,\ j\in J_{\alpha'}}\subseteq A_{G(\alpha)}^{\alpha}$,
    \item\label{consttwo} $x_{j}^{\alpha}\in\bigcap_{\alpha'<\alpha}A_{G(\alpha')}^{\alpha'}\smallsetminus C_{j}^{\alpha}$ for all $j\in J_\alpha$, and  
    
    \item\label{constfour} $x^{\alpha'}_{j'}\neq x^{\alpha}_j$ whenever $(\alpha',j')\neq(\alpha,j)$.
\end{enumerate}
We clarify that no $x^\alpha_j$ is defined when $J_\alpha=\emptyset$.

Assume we have already defined $G(\alpha')$ and $x_{j'}^{\alpha'}$ for any $\alpha'<\alpha$ and $j'\in J_{\alpha'}$. Choose a well-order $\lhd^\alpha$ of $J_\alpha$ and define $x^\alpha_j$ by recursion on $j$ (with respect to $\lhd^\alpha$). So assume that we have defined $x_{j'}^\alpha$ for all $j'\lhd^\alpha j$. Set 
\[B_j^\alpha:=C_{j}^\alpha\cup\set{x_{j'}^{\alpha'}}{ \alpha'<\alpha,\ j'\in J_{\alpha'}}\cup\{x^\alpha_{j'}:\, j'\in J_\alpha,\ j'\lhd^\alpha j\}.\]
The set $\set{x_{j'}^{\alpha'}}{ \alpha'<\alpha,\ j'\in J_{\alpha'}}\cup\{x^\alpha_{j'}:\, j'\in J_\alpha,\ j'\lhd^\alpha j\}$ has size ${\leq}\sum_{\alpha'\leq\alpha}|J_{\alpha'}|<\non(\SNwf)=\minnon$, so $B_j^\alpha\in\Iwf_{f_\alpha}$. By \autoref{def:I_fsystS}~\ref{it:I5}, $\bigcap_{\alpha'<\alpha}A_{G(\alpha')}^{\alpha'}\notin \Iwf_{f_\alpha}$, so there is some $x_{j}^{\alpha}\in\bigcap_{\alpha'<\alpha}A_{G(\alpha')}^{\alpha'}\smallsetminus B^\alpha_j$.

Note that $\set{x_{j}^{\alpha'}}{\alpha'\leq\alpha,\ j\in J_{\alpha'}}\in\Iwf_{f_\alpha}$, so there must be some $G(\alpha)\in I$ such that $\set{x_j^{\alpha'}}{ \alpha'\leq\alpha,\ j\in J_{\alpha'}}\subseteq A_{G(\alpha)}^{\alpha}$, and thus~\ref{constone} follows. Properties~\ref{consttwo} and~\ref{constfour} are clear. 

Now $K:=\set{x_{j}^{\alpha}}{ \alpha<\delta,\ j\in J_\alpha}\in\SNwf$ because it is contained in $\bigcap_{\gamma<\delta}A_{G(\gamma)}^{\gamma}$ by~\ref{constone} and~\ref{consttwo}.
Also $x_{j}^{\alpha}\notin C_{j}^{\alpha}$ for any $\alpha<\delta$ and $j\leq J_\alpha$, so~\ref{itI} follows. Property~\ref{it:II} is clear.
\end{proof}



We present the remaining main results of this section as consequences of the previous theorem. The following is \autoref{non<cof:I}.


\begin{theorem}\label{non<cof}
Assume $\DS(\delta)$. 
\begin{enumerate}[label = \rm (\alph*)]
    \item\label{it:<cof} If $\delta\leq\non(\SNwf)$ then $\delta<\cof(\SNwf)$.
    \item \label{it:non<cof} If $\cf(\non(\SNwf))=\cf(\delta)$ then $\non(\SNwf)<\cof(\SNwf)$ and there is some $K\in\SNwf$ of size $\non(\SNwf)$.
\end{enumerate}
\end{theorem}
\begin{proof}
%
%
Assume that either $\delta<\non(\SNwf)$ or $\cf(\non(\SNwf))=\cf(\delta)$.
In the first case, $\delta<\non(\SNwf)\leq\cof(\SNwf)$. 
In the second case, setting $\nu:=\non(\SNwf)$, we can easily find some $h\colon \delta\to \nu$ such that $\sum_{\alpha'<\alpha}|h(\alpha')|<\nu$ for all $\alpha<\delta$ and $\sup_{\alpha<\delta} h(\alpha)=\nu$ (allowing $h(\alpha)=0$ at many coordinates). If $\set{C_\xi}{\xi<\nu}\subseteq\SNwf$ then we can apply \autoref{SNpower} to $J_\alpha=h(\alpha)$ and $C^\alpha_j=C_j$ to find some $K\in\SNwf$ of size $\sum_{\alpha<\delta}|h(\alpha)|=\nu$ such that $K\nsubseteq C_\xi$ for all $\xi<\nu$, so the family $\set{C_\xi}{\xi<\nu}$ is not cofinal in $\SNwf$.
\end{proof}

As an immediate consequence we get:

\begin{corollary}\label{largecfSN}
If $\dfrak\leq\cof(\SNwf)$ then $\cov(\Mwf)<\cof(\SNwf)$.
\end{corollary}
\begin{proof}
In the case $\cov(\Mwf)=\dfrak$ the result follows by \autoref{lem:cov=d} and \autoref{non<cof} (recall that $\cov(\Mwf) \leq \non(\SNwf)$). In the case $\cov(\Mwf)<\dfrak$ the result follows by the hypothesis $\dfrak\leq\cof(\SNwf)$.
\end{proof}

It is not known whether $\dfrak\leq\cof(\SNwf)$ holds in ZFC.

\autoref{SNpower} can be used to find lower bounds of $\cof(\SNwf)$ under~$\DS(\delta)$ and some hypothesis on cardinal characteristics.

\begin{theorem}\label{lowerSN+}
Assume $\DS(\delta)$ and $\non(\SNwf)=\supcof=\mu$.
If 
$c$ is a cofinal subset of $\delta$ of order type $\tau\leq\cf(\mu)$,
then $\la \mu^{c},\leq\ra\leqT\SNwf$. In par\-ti\-cu\-lar $\add(\SNwf)\leq\cf(\mu)$ and $\dfrak^{|c|}_\mu\leq\cof(\SNwf)$.
\end{theorem}
\begin{proof}
It is enough to show that there are functions $\Psi_-\colon\mu^{c} \to\SNwf$ and $\Psi_+\colon\SNwf\to\mu^{c}$ such that, for any $z\in\mu^{c}$ and $B\in\SNwf$, if  $z\not\leq \Psi_+(B)$ then $\Psi_-(z)\not\subseteq B$.

Choose a witness $\bar f=\la f_\alpha:\, \alpha<\delta\ra$ of $\DS(\delta)$, and an $\bar f$-dominating system $\Seq{\bar A^\alpha}{\alpha<\delta}$ on $\mu$.

Fix $z\in\mu^{c}$. Since $\tau\leq\cf(\mu)$ and $c$ is cofinal in $\delta$, we have that $\sum_{\alpha'\in\alpha\cap c}|z(\alpha)|<\mu=\non(\SNwf)$ for any $\alpha<\delta$. Hence, we can apply \autoref{SNpower} to  $J_{\alpha}=z(\alpha)$ and $C^{\alpha}_j=A^{\alpha}_j$ for $\alpha\in c$, $J_\alpha=\emptyset$ otherwise, to obtain a 
set $\Psi_-(z)\in\SNwf$ such that $\Psi_-(z)\nsubseteq A_{\zeta}^{\alpha}$ for all $\alpha\in c$ and $\zeta< z(\alpha)$.
\footnote{If $c$ where not cofinal and bounded by some $\alpha<\mu$, it would be possible to have $\sum_{\alpha'<\alpha}|J_{\alpha'}| = \non(\SNwf)$ (when $\tau=\cf(\delta)$), in which case we cannot use \autoref{SNpower}.} 

For $B\in\SNwf$, choose $y_B\in\mu^\delta$ such that $B\subseteq\bigcap_{\alpha<\delta}A_{y_B(\alpha)}^{\alpha}$, and set $\Psi_+(B):=y_B\frestr c$. If $z\not\leq \Psi_+(B)$ then there is an $\alpha\in c$ such that $z(\alpha) > y_B(\alpha)$, so $\Psi_-(z)\nsubseteq A_{y_B(\alpha)}^{\alpha}$. Thus $\Psi_-(z)\nsubseteq B$ because $B\subseteq\bigcap_{\alpha<\delta}A_{y_B(\alpha)}^\alpha$.
\end{proof}

In the previous theorem, the case $\tau<\cf(\mu)$ is uninteresting because it implies $\la\mu^c,\leq\ra\eqT\mu$ by (the footnote of) \autoref{FacprodRS}~\ref{FacprodRSf}, in particular $\add(\SNwf)\leq\cf(\mu)=\dfrak^{|c|}_\mu\leq\cof(\SNwf)$. But the latter is already known, i.e.\ $\add(\Iwf)\leq\cf(\non(\Iwf))\leq\non(\Iwf)\leq\cof(\Iwf)$ for any ideal $\Iwf$. The case $\tau=\cf(\mu)$ is interesting, it implies $\la\mu^c,\leq\ra\eqT\la \tau^\tau,\leq\ra$ and gives us \autoref{lowerSN:I} as a consequence.

\begin{theorem}\label{lowerSN}
Assume $\DS(\delta)$ and $\non(\SNwf)=\supcof=\mu$.
If $\lambda:=\cf(\delta)=\cf(\mu)$,
then $\la \lambda^\lambda,\leq\ra\leqT\SNwf$. In par\-ti\-cu\-lar $\add(\SNwf)\leq\lambda$ and $\dfrak_\lambda\leq\cof(\SNwf)$. Moreover, $\dfrak_\lambda\neq\mu$ and $\mu<\cof(\SNwf)$.
\end{theorem}
\begin{proof}
The Tukey connection is a direct consequence of \autoref{lowerSN+}. Since $\cf(\dfrak_\lambda)\geq\bfrak_\lambda>\lambda$ and $\cf(\mu)=\lambda$, we must have $\dfrak_\lambda\neq\mu$. On the other hand, $\mu<\cof(\SNwf)$ follows by \autoref{non<cof}.
\end{proof}


\begin{corollary}\label{cor:lowSN}
Assume $\cov(\Mwf)=\dfrak$ and $\non(\SNwf)=\supcof=\mu$. If $\lambda:=\cf(\dfrak)=\cf(\mu)$ then $\lambda^\lambda\leqT\SNwf$. Moreover, $\dfrak_\lambda\neq\mu$ and $\mu<\cof(\SNwf)$.
\end{corollary}

Yorioka's characterization of $\cof(\SNwf)$ follows as a consequence of the main results of this section.

\begin{corollary}[Yorioka~{\cite{Yorioka}}]\label{YchaSN}
If $\minadd=\supcof=\lambda$ then $\SNwf\eqT\la\lambda^\lambda,\leq\ra$. In particular, $\add(\SNwf)=\lambda$ and $\cof(\SNwf)=\dfrak_\lambda$.
\end{corollary}
\begin{proof}
The connection $\la\lambda^\lambda,\leq\ra\leqT\SNwf$ follows by \autoref{cor:lowSN} because the cardinal characteristics $\cov(\Mwf)$, $\dfrak$ and $\non(\SNwf)$ are between $\minadd$ (which is regular) and $\supcof$. On the other hand, by~\autoref{new_upperb}, $\SNwf\leqT\Cbf_{[\supcof]^{<\minadd}}^\dfrak=\Cbf^\lambda_{[\lambda]^{<\lambda}}\eqT\lambda^\lambda$. Therefore, $\SNwf\eqT\la\lambda^\lambda,\leq\ra$.
\end{proof}

Our results also imply \autoref{thm:C3.12} as a corollary. Additionally, we can also conclude in this theorem that $\lambda$ is regular and $\dfrak=\non(\SNwf)=\supcof=\lambda$.

\begin{proof}[Proof of \autoref{thm:C3.12}.]
Assume that $\kappa$ and $\lambda$ are cardinals such that $0<\kappa\leq\lambda\leq\non(\SNwf)$, $\bar f=\la f_\alpha:\, \alpha<\lambda\ra$ is dominating in $\baireincr$ and that there is some $\bar f$-dominating directed system on $\kappa\times\lambda$. Then, for each $\alpha<\lambda$, $\Iwf_{f_\alpha}\leqT \kappa\times\lambda$, so $\cof(\Iwf_{f_\alpha})\leq\dfrak(\kappa\times\lambda) =\max\{ \cf(\kappa),\cf(\lambda)\}$. Therefore $\supcof\leq \max\{ \cf(\kappa),\cf(\lambda)\}$, and since $\non(\SNwf)=\minnon\leq \supcof$, $\kappa\leq\lambda \leq \max\{ \cf(\kappa),\cf(\lambda)\}$. The latter implies that $\lambda$ is regular, so $\non(\SNwf)=\supcof=\lambda$. On the other hand, the existence of the $\bar f$-dominating system above indicates that $\DS(\lambda)$ holds, so the conditions of \autoref{lowerSN} are met, and we obtain $\lambda^\lambda\leqT \SNwf$. Furthermore, by \autoref{cor:DSreg}, $\dfrak=\lambda$.
\end{proof}


\section{Preserving the covering of strong measure zero}\label{CovSN}

In this section, as mentioned in the introduction, we generalize Pawlikowki's result~\cite{P90} and show that certain iterations of precaliber $\theta$ posets force $\cov(\SNwf)\leq\theta$, when $\theta$ is regular (\autoref{thm:precaliberI}). We dissect the elements of Pawlikowski's original proof and present further results.

Below, we introduce the main element of Pawlikowski's proof (originally for $\theta=\aleph_1$).
Inspired by the notion of \emph{Rothberger family} from~\cite{Cichon,BS22}, we use the name \emph{$\theta$-Rothberger sequence}.








\begin{definition}\label{NicePaw}
Let $\theta$ be an infinite cardinal and let $I$ be a set. Let $\Hor:=\Seq{\Hwf_i}{i\in I}$ be a sequence of non-empty families of open sets in $2^\omega$. 
\begin{enumerate}[label = \rm (\arabic*)]
    \item\label{NicePaw1} We say that $\Hor$ is \emph{$\theta$-Rothberger} if it fulfills 
\[\forall\, J\in[I]^{\theta}\colon \set{\bigcap\Hwf_j}{j\in J} \text{ covers }2^\omega.\] 
When $\theta=\aleph_1$, we say that $\Hor$ is a \emph{Rothberger sequence}.
    \item\label{NicePaw2} We say that $\Hor$ \emph{produces a strong measure zero set} if
    \[\set{\hgt_\sigma}{\sigma\in(2^{<\omega})^\omega,\ \exists\, i\in I\colon \bigcup_{n<\omega}[\sigma_n]\in\Hwf_i}\]
    is dominating in $\omega^\omega$. It is clear that, in this case, $\bigcap_{i\in I}(\bigcap\Hwf_i)\in\SNwf$.
\end{enumerate}
Note that, whenever $\Hor$ is $\theta$-Rothberger and $A\subseteq I$,  $\Hor\frestr A:=\Seq{\Hwf_i}{i\in A}$ is also $\theta$-Rothberger.

We remark that the interpretation of $\Hor$ in any model depends on the interpretation of the open sets in $\bigcup_{i\in I}\Hwf_i$.
\end{definition}



We point out below some characterizations of $\theta$-Rothberger sequences.

\begin{fact} Let $\theta$ be an infinite cardinal, let $I$ be a set and  let $\Hor$ be as in~\autoref{NicePaw}. Then the following statements are equivalent. 
\begin{enumerate}[label=\rm (\roman*)]
    \item $\Hor$ is $\theta$-Rothberger.
    
    \item $\forall x\in2^\omega\colon |\set{i\in I}{x\notin\bigcap\Hwf_i}|<\theta$.
    
    \item $\forall\, \bar U\in\prod_{i\in I}\Hwf_i\ \forall\, J\in[I]^{\theta}\colon \bigcup_{i\in J}U_i=2^\omega$.

    \item $\forall\, x\in2^\omega\ \forall\, \bar U\in\prod_{i\in I}\Hwf_i\colon |\set{i\in I}{x\notin U_i}|<\theta$.
\end{enumerate}
\end{fact}

Potential Rothberger sequences producing strong measure zero sets are obtained using sequences of posets as follows.

\begin{definition}\label{defcovp}
Let $\pi$ be a limit ordinal and let
$\tbf=\la \Por_\xi:\, \xi\leq\pi\ra$ be a sequence of posets such that $\Por_\xi$ is a complete subset of $\Por_\eta$ for $\xi\leq\eta\leq\pi$, and $\Por_\pi=\bigcup_{\xi<\pi}\Por_\xi$ (direct limit).
For $\xi<\pi$, define
\[\Hwf_\xi^\tbf:=\set{D\subseteq2^\omega }{ D \textrm{\ is an open set coded in } V_{\xi+1} \textrm{\ such that\ } 2^\omega\cap V_{\xi}\subseteq D}\]
and let $H_\xi^\tbf:=\bigcap\Hwf_\xi^\tbf$. 
Here $V_\eta:=V^{\Por_\eta}$ for any $\eta\leq\pi$. Also, define
$\Hor^\tbf:=\Seq{\Hwf_\xi^\tbf}{\xi<\pi}$.

For a set $A\subseteq\pi$, let us define
\[E^\tbf(A):=\bigcap_{\xi\in A}H_\xi^\tbf\]
Notice that $2^\omega\cap V_{\min A}\subseteq E^\tbf(A)$ when $A\neq\emptyset$.
\end{definition}

It is clear that $2^\omega\cap V_\xi\subseteq H_\xi^\tbf$ for $\xi<\pi$. Moreover, $V_{\xi+1}\models H_\xi^\tbf=2^\omega\cap V_\xi$ because, for any $x\in 2^\omega\cap V_{\xi+1}\menos V_\xi$, $2^\omega\menos\{x\}\in \Hwf^\tbf_\xi$. However, it could happen that $V_{\xi+2}\models 2^\omega\cap V_\xi \subsetneq H_\xi^\tbf$ because there could be a Cohen real in $2^\omega\cap V_{\xi+2}$ over $V_{\xi+1}$, so it belongs to $H_\xi^\tbf$.



 The main point of the previous definition is illustrated by the following result (compare with~\cite[Remark, p.~677]{P90}). 



\begin{theorem}\label{lemcovp}
Let $\tbf$ be a sequence of posets as in \autoref{defcovp} such that $\cf(\pi)>\omega$, $\Por_\pi$ has the $\cf(\pi)$-cc, and $\Por_{\xi+1}$ adds a Cohen real over $V_\xi$ for all $\xi<\pi$. Then, in $V_\pi$:
\begin{enumerate}[label = \rm (\alph*)]
    \item\label{lemcovpb} If $A$ is cofinal in $\pi$, then $\Hor^\tbf\frestr A$ produces a strong measure zero set. In particular, $E^\tbf(A)\in\SNwf$
    \item\label{lemcovpc} $2^\omega\cap V_\xi\in\SNwf$ for any $\xi<\pi$.
    \item\label{lemcovpd} $\Cbf_\SNwf \leqT \pi$. In particular, $\cov(\SNwf) \leq \cf(\pi) \leq \non(\SNwf)$.
    \item\label{lemcovpa} If $L\subseteq\pi$ is cofinal with order type $\cf(\pi)$ then $\Hor^\tbf\frestr L$ is a $\cf(\pi)$-Rothberger sequence.
\end{enumerate}
\end{theorem}
\begin{proof}
\ref{lemcovpb}: Fix $f\in \omega^\omega\cap V_\pi$ increasing. Since, in $V$, $\cf(\pi)>\omega$ and $\Por_\pi$ has the $\cf(\pi)$-cc, and (in $V_\pi$) $A$ is cofinal in $\pi$, we can find some $\xi\in A$ such that $f\in V_{\xi}$.
There is some $\sigma_\xi\in V_{\xi+1}\cap \prod_{n<\omega}2^{f(n)}$ (so $\hgt_{\sigma_\xi}=f$) which is Cohen over $V_\xi$, so $D_\xi:=\bigcup_{i<\omega}[\sigma_\xi(i)]$ is in $\Hwf_\xi^\tbf$. Therefore, $\Hor\frestr A$ produces a strong measure zero set and $E^\tbf(A)\in\SNwf$.

\ref{lemcovpc}: Clear because $2^\omega\cap V_\xi\subseteq E^\tbf([\xi,\pi))$.

\ref{lemcovpd}: Define $\Psi_-:2^\omega\to\pi$ by choosing $\Psi_-(x)<\pi$ such that $x\in V_{\Psi_-(x)}$, and define $\Psi_+:\pi\to\SNwf$ by $\Psi_+(\xi):=2^\omega\cap V_\xi$ for $\xi<\pi$ (which is in $\SNwf$ by~\ref{lemcovpc}). It is clear that $(\Psi_-,\Psi_+)$ is the desired Tukey connection. 

\ref{lemcovpa}: Let $x\in 2^\omega\cap V_\pi$.  Then we get some $\xi\in L$ such that $x\in V_\xi$, so $x\in H^\tbf_\eta$ for all $\eta\in L$ with $\eta\geq \xi$, i.e.\ $|\{\xi\in L :\, x\notin H^\tbf_\xi\}|<\cf(\pi)$.
%
\end{proof}


\begin{corollary}\label{upbcov}
Let $\pi$ be a limit ordinal of uncountable cofinality and let $\Por_\pi=\la\Por_\xi,\Qnm_\xi:\,  \xi<\pi\ra$ be a FS iteration of non-trivial $\cf(\pi)$-cc posets. Then $\Por_\pi$ forces $\Cbf_{\SNwf}\leqT\pi$. In particular, $\Por_\pi$ forces $\cov(\SNwf)\leq\cf(\pi)\leq \non(\SNwf)$.\footnote{The second inequality is natural already because $\Por_\pi$ forces $\cf(\pi)\leq\cov(\Mwf)$ (even more, $\pi\leqT \Cbf_{\Mwf}$).}  
\end{corollary}
\begin{proof}
Apply \autoref{lemcovp} to $\tbf:=\la \Por_{\omega\xi}:\, \xi<\pi\ra$ (note that $\omega\pi=\pi$ because $\pi$ has uncountable cofinality, which is easy to check using the Cantor normal form of $\pi$), considering that $\Por_{\omega(\xi+1)}$ adds a Cohen real over $V^{\Por_{\omega\xi}}$ for all $\xi<\pi$.
%
%
%
%
%
%
%
%
\end{proof}

We now look at stronger Tukey connections obtained from Rothberger sequences.

\begin{lemma}\label{TukeySN}
    Let $\theta$ be an infinite cardinal and let $\Hor=\la \Hwf_i:\, i\in I\ra$ be a $\theta$-Rothberger family satisfying
    \begin{enumerate}[label = $(\otimes)$]
        \item\label{ot1} There is a pairwise disjoint family $\set{A_\ell}{\ell\in L}$ of subsets of $I$ such that $\Hor\frestr A_\ell$ produces a strong measure zero set.
    \end{enumerate}
    Then $\Cbf_{[L]^{<\theta}}\leqT \Cbf^\perp_{\SNwf}$. In particular, when $|L|\geq\theta$, $\cov(\SNwf)\leq\theta$ and $\cov([L]^{<\theta})\leq\non(\SNwf)$.
\end{lemma}

This result is uninteresting in the case $|L|<\theta$ because $\Cbf_{\pts(L)}\leqT\Rbf$ holds for any relational system $\Rbf$.

\begin{proof}
Define the functions $\Psi_-\colon L\to\SNwf$ and $\Psi_+\colon 2^\omega\to[L]^{<\theta}$ 
by $\Psi_-(\ell):=\bigcap_{i\in A_\ell}(\bigcap \Hwf_i)$ and $\Psi_+(x):=\set{\ell\in L}{x\not\in \Psi_-(\ell)}$. 
It is clear by~\ref{ot1} that $\Psi_-(\ell)\in\SNwf$ for all $\ell\in L$, but we must show that $|\Psi_+(x)|<\theta$ for all $x\in 2^\omega$, which in turn will guarantee that $(\Psi_-,\Psi_+)$ is the desired Tukey connection.

Let $x\in2^\omega$. For each $\ell\in\Psi_+(x)$ there is some $i_\ell\in A_\ell$ such that $x\notin \bigcap \Hwf_{i_\ell}$.
Since $\set{A_\ell}{\ell\in L}$ is pairwise disjoint, we get that $C:=\set{i_\ell}{\ell\in\Psi_+(x)}$ has the same size as $\Psi_+(x)$ and $\set{\bigcap\Hwf_i}{i\in C}$ does not cover $2^\omega$, thus $|C|<\theta$ because
 $\Hor$ is $\theta$-Rothberger.
\end{proof}


For the remaining of this section, fix an uncountable regular cardinal $\theta$. Recall that a poset $\Por$ \emph{has precaliber $\theta$} if, for any $A\in [\Por]^\theta$, there is some centered $B\in[A]^\theta$, i.e. any finite $F\subseteq B$ has a common stronger condition in $\Por$.
Pawlikowski~\cite{P90} proved that any Rothberger sequence is preserved by precaliber $\aleph_1$ forcing notions. We generalize this and show that the $\theta$-Rothberger property is preserved by $\theta$-precaliber forcing notions.

\begin{lemma}\label{lem:precaliber}
Any poset with precaliber $\theta$ preserves all $\theta$-Rothberger sequences from the ground model. 
\end{lemma}
\begin{proof}
Let $\Por$ be a $\theta$-precaliber poset and let $\Hor$ be a $\theta$-Rothberger sequence as in \autoref{NicePaw}. Towards a contraction, assume that there are $\Por$-names $\dot x$ and $\dot J$, and a condition $p\in\Por$ such that  
\[p\Vdash \dot J\in[I]^\theta\textrm{\ and\ }\dot x\notin \bigcup_{j\in \dot J}\bigcap \Hwf_j.\]
Find an $A_0\in[I]^\theta$ and a sequence $\Seq{(p_j,D_j)}{j\in A_0}$
such that $p_j\leq p$ in $\Por$, $D_j\in\Hwf_j$ and $p_j\Vdash $``$j\in\dot J$ and  $\dot x\notin D_{j}$" for all $j\in A_0$. Since $\Por$ has precaliber $\theta$, there is some $A\in[A_0]^\theta$ such that $\set{ p_j}{j\in A}$ is centered.

Let $F\subseteq A$ be a finite set. Since $\set{ p_j}{j\in A}$ is centered, choose some $q\in\Por$ such that $q\leq p_j$ for all $j\in F$. Then,
\[q\Vdash \dot x\not\in\bigcup_{j\in F}D_j,\]
so $\bigcup_{j\in F}D_j\neq2^\omega$ (in the ground model).
This means that $2^\omega$ cannot be covered by finite sets from $\set{D_j}{j\in A}$, so $\bigcup_{j\in A}D_j\neq2^\omega$ by compactness of $2^\omega$. Therefore 
\[\bigcup_{j\in A}\bigcap \Hwf_j \neq 2^\omega,\]
which contradicts that $\Hor$ is $\theta$-Rothberger.
\end{proof}


We conclude with the main result of this section (\autoref{thm:precaliberI}). 

\begin{theorem}[{\cite{P90}}]\label{thm:precaliberstr}
Assume that $\theta>\aleph_0$ is regular. 
Let $\pi$ be a limit ordinal with $\cf(\pi)>\aleph_0$ and let
$\tbf=\la \Por_\xi:\, \xi\leq\pi\ra$ be a sequence of posets satisfying: 
\begin{enumerate}[label = \rm(\roman*)]
    \item\label{pc1} $\Por_\xi$ is a complete subset of $\Por_\eta$ for $\xi\leq\eta\leq\pi$,
    \item\label{pc2} $\Por_\pi=\bigcup_{\xi<\pi}\Por_\xi$ (direct limit) and it has $\cf(\pi)$-cc,
    \item\label{pc3} $\Por_\eta=\bigcup_{\xi<\eta}\Por_\xi$ for any $\eta\leq\pi$ with $\cf(\eta)=\theta$, and it has $\theta$-cc,
    \item\label{pc4} $\Por_\eta$ forces that $\Por_\pi/\Por_\eta$ has precaliber $\theta$ whenever $\eta\leq\pi$ with $\cf(\eta)=\theta$, and
    \item\label{pc5} for any $\xi<\pi$, $\Por_{\xi+1}$ adds a Cohen real over $V_\xi$.
\end{enumerate}
Let $\lambda:=\scf(\pi)$. Then $\Por_\pi$ forces: 
\begin{enumerate}[label = \rm (\alph*)]
    \item\label{precSNa} $\Hor^\tbf$ is $\theta$-Rothberger.
    \item\label{precSNb} $\Cbf_{[\lambda]^{<\theta}}\leqT\Cbf_{\SNwf}^{\perp}$. In particular, whenever $\scf(\pi)\geq\theta$, $\Por_\pi$ forces $\cov(\SNwf)\leq\theta$ and $\scf(\pi)\leq\non(\SNwf)$.
\end{enumerate}
\end{theorem}
\begin{proof}
By~\autoref{lem:scf}, choose a pairwise disjoint family $\set{A_\alpha}{\alpha<\lambda}$ of cofinal of subsets of $\pi$. By~\ref{pc1},~\ref{pc2},~\ref{pc5} and \autoref{lemcovp}, this family satisfies~\ref{ot1} of \autoref{TukeySN} for $\Hor^\tbf$ (in $V_\pi$). Thus~\ref{precSNb} follows from~\ref{precSNa}. So it is enough to prove~\ref{precSNa}.

Assume towards a contradiction that, in $V_\pi$, $\Hor^\tbf$ is not $\theta$-Rothberger, which means that there is some $C'\subseteq \pi$ of size $\theta$ such that
$\bigcup_{\xi\in C'}H_\xi^\tbf\neq2^\omega$.
Let $\eta\leq\pi$ be the minimum ordinal such that there is some $C\in[\eta]^\theta$ satisfying $\bigcup_{\xi\in C}H_\xi^\tbf\neq2^\omega$. Since $C$ must have order-type $\theta$, we have that $\cf(\eta)=\theta$. 

Notice that $\Hor^\tbf\frestr \eta$ is coded in $V_\eta$. By the previous paragraph, $\Hor^\tbf\frestr \eta$ is not $\theta$-Rothberger in $V_\pi$, so by~\ref{pc4} and \autoref{lem:precaliber} we have that, in $V_\eta$, this sequence is not $\theta$-Rothberger. This means that, in $V_\eta$, there is a $B\in[\eta]^\theta$ such that \[\bigcup_{\xi\in B}H_\xi^\tbf\neq2^\omega.\]
By the minimality of $\eta$, $B$ must be cofinal in $\eta$, and also $|B|=\theta=\cf(\eta)$, so by~\ref{pc4} and \autoref{lemcovp}~\ref{lemcovpa} $\Hor^\tbf\frestr B$ is $\theta$-Rothberger in $V_\eta$, contradicting that $\{H^\tbf_\xi:\, \xi\in B\}$ does not cover $2^\omega$.
\end{proof}


As a consequence, we have the following result for FS iterations.

\begin{corollary}[{\cite{P90}}]\label{thm:precaliber}
Assume that $\theta\geq\aleph_1$ is regular. Let
$\Por_\pi=\la \Por_\xi,\Qnm_\xi:\, \xi<\pi\ra$ be a FS iteration of non-trivial precaliber $\theta$ posets such that $\cf(\pi)>\omega$ and $\Por_\pi$ has $\cf(\pi)$-cc,
and let $\lambda:=\scf(\pi)$. Also let $\tbf=\la \Por_{\omega\xi}:\, \xi\leq \pi\ra$. Then $\Por_\pi$ forces: 
\begin{enumerate}[label = \rm (\alph*)]
    \item $\Hor^\tbf$ is $\theta$-Rothberger.
    \item $\Cbf_{[\lambda]^{<\theta}}\leqT\Cbf_{\SNwf}^{\perp}$. In particular, whenever $\scf(\pi)\geq\theta$, $\Por_\pi$ forces $\cov(\SNwf)\leq\theta$ and $\scf(\pi)\leq\non(\SNwf)$.
\end{enumerate}
\end{corollary}
\begin{proof}
    It is enough to show that $\tbf$ satisfies~\ref{pc1}--\ref{pc5} of \autoref{thm:precaliberstr}. \ref{pc1},~\ref{pc2} and the first equality in~\ref{pc3} are clear because the iteration has finite support, and $\Por_\eta$ has precaliber $\theta$ for all $\eta\leq\pi$ because any FS iteration of precaliber $\theta$ posets has precaliber $\theta$ when $\theta$ is regular, thus it has $\theta$-cc and~\ref{pc3} holds. For the same reason,~\ref{pc4} holds. Finally,~\ref{pc5} holds because FS iterations add Cohen reals at limit steps.
\end{proof}

\section{FS iterations and dominating systems}\label{Sec:eff}

The purpose of this section is to provide applications of our results in forcing iterations. In particular, we
perform a forcing construction by using the matrix iteration technique from~\cite{mejiavert} to prove~\autoref{mainappl}. In addition, we show how optimal are the assumptions of the main results in \autoref{BounSN}.

For many iterations adding Cohen reals, we can use an argument as in the proof of \autoref{lem:DomSys} to show that they force $\DS(\delta)$ for some $\delta$. 



\begin{theorem}\label{thm:DS+}
 Let $\kappa\leq\lambda$ be uncountable cardinals with $\kappa$ regular and $\lambda^{<\kappa} = \lambda$, $\pi=\lambda\gamma$ where $0<\gamma<\lambda^+$ is an ordinal and $\cf(\pi)\geq\kappa$ (i.e.\ either $\gamma$ is a successor ordinal or $\cf(\gamma)\geq \kappa$), and let 
 $\la \Por_\xi:\, \xi\leq\pi\ra$ be a sequence of posets such that
 \begin{enumerate}[label =\rm (\roman*)]
    \item\label{DS+1} $\Por_\xi$ is a complete subset of $\Por_\eta$ for $\xi\leq\eta\leq\pi$,
    \item\label{DS+2} $\Por_\pi=\bigcup_{\xi<\pi}\Por_\xi$ and it has $\kappa$-cc,
    \item\label{DS+3} $|\Por_\pi| \leq \lambda$ and
    \item\label{DS+4} for any $\xi<\pi$, $\Por_{\xi+1}$ adds a Cohen real over $V_\xi$.
 \end{enumerate}
 Then $\Por_\pi$ forces $\cfrak=\lambda$ and $\DS(\lambda\beta)$ whenever $\beta=\delta+\rho$ for some $\delta<\gamma$ and $\rho\leq\ir(\gamma)$ with $\cf(\lambda\rho)=\cf(\pi)$.
 
As a consequence, $\Por_\pi$ forces $\DS(\lambda\cf(\gamma))$ (which is $\DS(\lambda)$ when $\gamma$ is successor).
\end{theorem}
\begin{proof}
Since $\lambda^{<\kappa}=\lambda$ and $\Por_\pi$ has $\kappa$-cc and size ${\leq}\lambda$, it is clear that $\Por_\pi$ forces $\cfrak=\lambda$, so it remains to prove that $\Por_\pi$ forces $\DS(\lambda\beta)$ for any $\beta$ as in the hypothesis. 

By~\autoref{lemincseq}, choose an increasing cofinal sequence $\Seq{\zeta_\eta}{\eta<\lambda\beta}$ in $\pi$. To force $\DS(\lambda\beta)$, by recursion on $\eta<\lambda\beta$ we are going to define $\Por_\pi$-names $\Seq{\dot f_\eta}{\eta<\lambda\beta}\subseteq\baireincr$ and $\Seq{\dot A_\eta}{ \eta<\lambda\beta}$ for a dominating family and dense $G_\delta$ sets in $2^\omega$, respectively, such that $\Por_\pi$  forces
 $\dot A_\eta\in \Iwf_{\dot f_\eta}$ and $\bigcap_{\xi<\eta} \dot A_\xi\notin \Iwf_{\dot f_\eta}$.

Before starting the recursion, by counting arguments (like book-keeping), we can find a sequence $\la\dot h_\eta:\, \eta<\lambda\beta\ra$ of nice $\Por_\pi$-names enumerating all $\omega^\omega\cap V_\pi$ such that each $\dot h_\eta$ is a $\Por_{\zeta_\eta}$-name for $\eta<\lambda\beta$.

Assume that we have constructed the $\Por_{\zeta_\eta}$-names 
$\Seq{\dot f_\xi}{ \xi<\eta}$ and $\Seq{ \dot A_\xi}{\xi<\eta}$. As in the proof of~\autoref{lem:DomSys}, we can find a $\Por_{\zeta_{\eta+1}}$-name $\dot P$ for a perfect set of Cohen reals over $V^{\Por_{\zeta_\eta}}$ (because 
$\Por_{\zeta_\eta+1}$, and hence $\Por_{\zeta_{\eta+1}}$, adds a Cohen real over $V_{\zeta_\eta}$), so $\dot P\subseteq \bigcap_{\xi<\eta}\dot A_{\xi}$ and $\dot P\notin\SNwf$ are forced by $\Por_{\zeta_{\eta+1}}$. Hence, there is a $\Por_{\zeta_{\eta+1}}$-name $\dot g$ for a member of $\baire$ such that $\Por_{\zeta_{\eta+1}}$ forces $\bigcap_{\xi<\eta}\dot A_{\xi}\notin\Iwf_{\dot g}$.
On the other hand, we can find a $\Por_{\zeta_{\eta+1}}$-name $\dot f_\eta$ for a member of $\baire$ such that $\Por_{\zeta_{\eta+1}}$ forces $\dot f_\eta\geq^* \dot h_\eta$ and $\dot f_\eta\geq^* \dot g$, and a 
$\Por_{\zeta_{\eta+1}}$-name for a dense $G_\delta$ set $\dot A_\eta\in\Iwf_{\dot f_\eta}$. This finishes the construction.
\end{proof}

Since FS iterations add Cohen reals at limit steps, we have the following reformulation of the previous theorem in terms of FS iterations.

\begin{corollary}\label{thm:DS}
 Let $\kappa\leq\lambda$ be uncountable cardinals with $\kappa$ regular and $\lambda^{<\kappa} = \lambda$, let $\pi=\lambda\gamma$ where $0<\gamma<\lambda^+$ is an ordinal and $\cf(\pi)\geq \kappa$, and let $\Por_\pi=\la\Por_\xi,\Qnm_\xi:\, \xi<\pi\ra$ be a FS iteration of non-trivial $\kappa$-cc posets of size ${\leq}\lambda$. Then $\Por_\pi$ forces $\cfrak=\lambda$ and $\DS(\lambda\beta)$ whenever $\beta=\delta+\rho$ for some $\delta<\gamma$ and $\rho\leq\ir(\gamma)$ with $\cf(\lambda\rho)=\cf(\pi)$.
 
As a consequence, $\Por_\pi$ forces $\DS(\lambda\cf(\gamma))$. 
\end{corollary}
\begin{proof}
Apply \autoref{thm:DS+} to $\la \Por_{\omega\xi}:\, \xi\leq\pi\ra$ (and use \autoref{lemincseq}).
\end{proof}

Now we present some consequences of the foregoing theorem. In particular, $\cov(\SNwf)<\non(\SNwf)<\cof(\SNwf)$ is already true in Cohen model. Denote by $\Cor_\lambda$ the poset adding $\lambda$-many Cohen reals.

\begin{corollary}\label{Cmodel}
Assume that $\lambda^{\aleph_0}=\lambda$. Then
$\Cor_\lambda$ forces $\DS(\lambda\beta)$ for all $\beta<\lambda^+$ such that either $\beta$ is successor or $\cf(\beta)>\omega$. In particular, $\dfrak_{\cf(\lambda)}\leq \cof(\SNwf) \leq \cov\left(([\lambda]^{<\aleph_1})^\lambda\right)$ and $\lambda<\cof(\SNwf)$. Besides $\cov(\SNwf)=\aleph_1$ and $\non(\SNwf)=\lambda$.
\end{corollary} 
\begin{proof}
 It is well-known that $\Cor_\lambda$ forces $\add(\Nwf)=\minadd=\non(\Mwf)=\aleph_1$ and $\cov(\Mwf)=\supcof=\cof(\Nwf)=\cfrak=\lambda$. On the other hand, by \autoref{thm:precaliber} applied to $\theta=\aleph_1$ (i.e. Pawlikoski's original result), $\Cor_\lambda$ forces $\cov(\SNwf)=\aleph_1$.
 

 Let $\beta<\lambda^+$ be an ordinal as in the hypothesis.  
 Since $\Cor_\lambda\cong \Cor_{\lambda\beta}$, by \autoref{thm:DS}
 we get that $\Cor_\lambda$ forces $\DS(\lambda\beta)$. Next, by applying~\autoref{new_upperb} and~\ref{lowerSN}, $\Cor_\lambda$ forces that $\dfrak_{\cf(\lambda)}\leq\cof(\SNwf) \leq \cov\left(([\lambda]^{<\aleph_1})^\lambda\right)$ and $\lambda<\cof(\SNwf)$. 
\end{proof}

Hechler model also satisfies $\cov(\SNwf)<\non(\SNwf)<\cof(\SNwf)$.

\begin{corollary}\label{Dmodel}
Assume $\aleph_1\leq\kappa\leq\lambda=\lambda^{\aleph_0}$ with $\kappa$ regular. Let $\Dor_{\pi}$ be a FS iteration of Hechler forcing of length $\pi=\lambda\kappa$. Then, in $V^{\Dor_{\pi}}$, $\cov(\SNwf)=\aleph_1$, $\add(\Mwf)=\cof(\Mwf)=\kappa$ and $\non(\SNwf)=\cfrak=\lambda$. In particular, $\DS(\kappa)$ and $\DS(\lambda\kappa)$ hold in the extension and, if $\cf(\lambda)=\kappa$, then $\dfrak_\kappa\leq\cof(\SNwf)\leq \cov\left(([\lambda]^{<\aleph_1})^\kappa\right)$ and $\lambda<\cof(\SNwf)$. 
\end{corollary} 
\begin{proof}
It is well-known that $\Dor_{\pi}$ forces $\add(\Nwf)=\minadd=\cov(\Nwf)=\aleph_1$ and $\add(\Mwf)=\cof(\Mwf)=\kappa$ and $\non(\SNwf)=\cfrak=\lambda$ (see e.g.~\cite[Thm.~5]{mejiamatrix}). In addition, $\Dor_{\pi}$ forces $\cov(\SNwf)=\aleph_1$ by employing~\autoref{thm:precaliber} because Hechler forcing has precaliber $\aleph_1$.  

On the other hand, we obtain $\DS(\kappa)$ and $\DS(\lambda\kappa)$ by making use of~\autoref{lem:cov=d} and~\autoref{thm:DS}, respectively. The rest of the result follows by \autoref{new_upperb} and~\ref{lowerSN}.
\end{proof}

\begin{corollary}\label{Amodel}
Assume $\aleph_1\leq\kappa\leq\lambda=\lambda^{\aleph_0}$ with $\kappa$ regular.
Let $\Aor_{\pi}$ be the FS iteration of amoeba forcing of length $\pi=\lambda\kappa$. Then $\Aor_{\pi}$ forces $\add(\Nwf)=\cof(\Nwf)=\kappa$, $\lambda=\cfrak$, that
$\DS(\kappa)$ and $\DS(\lambda\kappa)$ hold, and  $\cof(\SNwf)=\dfrak_\kappa$.
\end{corollary} 
\begin{proof}
It is well-known that $\Aor_{\pi}$ forces $\add(\Nwf)=\cof(\Nwf)=\kappa$ and $\lambda=\cfrak$. 
By applying~\autoref{lem:cov=d}, we obtain that $\Aor_{\pi}$ forces $\DS(\kappa)$, consequently, it is forced that $\cof(\SNwf)=\dfrak_\kappa$ by~\autoref{YchaSN} (Yorioka's Theorem) because $\minadd=\supcof=\kappa$. Finally, by applying~\autoref{thm:DS}, we obtain that $\Aor_{\pi}$ forces $\DS(\lambda\kappa)$. 
\end{proof}

It is possible to decide the value of $\cof(\SNwf)$ in the previous corollaries by preserving the values of cardinals of the form $\dfrak_\kappa$ and $\cof\left(([\lambda]^{<\kappa})^w\right)$ from the ground model.

\begin{lemma}\label{lem:prcof}
    Assume that $\kappa$ is an uncountable regular cardinal and that $\Por$ is a $\kappa$-cc poset. If $\lambda\geq\kappa$ is a cardinal with cofinality ${\geq}\kappa$ and $w$ is a set, then $\Por$ preserves the ground-model value of $\cof([\lambda]^{<\kappa})$, $\cof\left(([\lambda]^{<\kappa})^w\right)$ and $\dfrak_\kappa$.
\end{lemma}
\begin{proof}
    The preservation of $\cof([\lambda]^{<\kappa})$ and $\cof\left(([\lambda]^{<\kappa})^w\right)$ is clear because, for any $\Por$-name $\dot a$ of a member of $[\lambda]^{<\kappa}$, by $\kappa$-cc we can find some $b\in [\lambda]^{<\kappa}$ (in the ground model) which is forced to contain $\dot a$. For $\dfrak_\kappa$, note that $\kappa^\kappa \eqT ([\kappa]^{<\kappa})^\kappa$.
\end{proof}

For example, to decide a value for $\cof(\SNwf)$ in \autoref{Cmodel}, assume that $\lambda^{\aleph_0}=\lambda$, $\cf(\lambda)^{<\cf(\lambda)}=\cf(\lambda)$ and $\mu^\lambda=\mu$. Recall that, in ZFC,
\[\dfrak_{\cf(\lambda)}=\dfrak^{\cf(\lambda)}_\lambda\leq \dfrak_\lambda \leq\cov\left(([\lambda]^{<\aleph_1})^\lambda\right)\leq \cof\left(([\lambda]^{<\aleph_1})^\lambda\right) \leq 2^\lambda,\]
even more, the equality $\lambda^{\aleph_0} = \lambda$ implies that the covering and cofinality numbers above are equal (because $[\lambda]^{<\aleph_1}\eqT \Cbf_{[\lambda]^{<\aleph_1}}$).
So, by forcing with $\Fn_{<\cf(\lambda)}(\mu,2)$ (the poset of partial functions from $\mu$ into $2$ with domain of size ${<}\cf(\lambda)$), we get in the generic extension that still $\lambda^{\aleph_0} = \lambda$, but $\dfrak_{\cf(\lambda)}= 2^\lambda =\mu$. Therefore, by forcing with $\Cor_\lambda$ afterwards, we obtain the statements in \autoref{Cmodel} plus $\cof(\SNwf)=\dfrak_{\cf(\lambda)}=2^\lambda = \mu$: $\dfrak_{\cf(\lambda)}=\mu$ by \autoref{lem:prcof}, and $2^\lambda = \mu$ can be calculated by counting nice-names.


\begin{remark}\label{optimal}
    The previous examples (and some more) help us to check how optimal are the hypotheses (and even the conclusions) of the main results in \autoref{BounSN}, which we list below.
    \begin{enumerate}[label = \rm(\arabic*)]
        \item\label{optimala} \autoref{Cmodel} tells us that $\DS(\delta)$ can hold for multiple $\delta$'s with different cofinalities.

        \item\label{optimalb} \autoref{Amodel} indicates that it is consistent to have $\DS(\delta)$ with $\delta>\cof(\SNwf)$ (by forcing with $\Aor_\pi$ on a model where $\kappa<\dfrak_\kappa<\lambda$, and setting $\delta:=\lambda\kappa$), which shows the need of the hypothesis $\delta\leq\non(\SNwf)$ in \autoref{non<cof}~\ref{it:<cof}.

        \item\label{delta0<cof} We do not know whether $\delta_0\leq\cof(\SNwf)$ holds where $\delta_0$ is the smallest ordinal $\delta$ satisfying $\DS(\delta)$ (if it exists). On the other hand, it is consistent that $\non(\SNwf)<\delta$ for all $\delta$ satisfying $\DS(\delta)$: any FS iteration of ccc posets forcing $\non(\SNwf)<\dfrak$ witnesses this, e.g.\ after a FS iteration of length $\lambda\kappa$ of random forcing, with $\kappa<\lambda$ and $\kappa$ uncountable regular. When $\lambda^{\aleph_0}=\lambda$, this also forces $\DS(\lambda\kappa)$ by \autoref{thm:DS}.

        \item\label{optimalc} Goldstern, Judah and Shelah~\cite{GJS} constructed, via a countable support iteration of $\omega^\omega$-bounding proper posets, a model where $\dfrak=\aleph_1$, $\cfrak=\aleph_2$ and $\SNwf = [2^\omega]^{<\cfrak}$. Then $\add(\SNwf)=\cof(\SNwf)=\aleph_2$, and $\DS(\omega_1)$ holds (by \autoref{lem:cov=d}). This shows the need for the hypothesis $\cf(\non(\SNwf))=\cf(\delta)$ in \autoref{non<cof}~\ref{it:non<cof} and also in \autoref{lowerSN} (and in \autoref{cor:lowSN}), the latter because $\dfrak_{\omega_1}>\aleph_2$ is possible (if assumed in the ground model, it is preserved, see~\cite[Sec.~3.3]{CMR}).


        \item\label{optimale} In the conclusion $\dfrak_\lambda\neq \mu$ of \autoref{lowerSN} (and \autoref{cor:lowSN}), it cannot be decided which is larger. For example, in \autoref{Dmodel} one of $\dfrak_\kappa<\lambda$ and $\lambda<\dfrak_\kappa$ can be assumed in the ground model, and the correct inequality is preserved in the $\Dor_\pi$-extension by \autoref{lem:prcof}.

        \item\label{optimalf} In all the examples so far we have that some $\DS(\delta)$ holds for some $\delta$ and $\cov(\Mwf)=\dfrak$, but it is consistent that $\cov(\Mwf)<\dfrak$ and that $\DS(\delta)$ holds for some $\delta$. Any FS iteration forcing $\cov(\Mwf)<\dfrak$ does this, e.g. after iterating $\Eor$ (the standard ccc poset adding an eventually different real) of length $\lambda\kappa$ (assuming $\kappa<\lambda=\lambda^{\aleph_0}$ and $\kappa$ uncountable regular, see e.g.~\cite[Thm.~2']{Br} and~\cite[Thm.~3]{mejiamatrix}).
    \end{enumerate}
\end{remark}

The consistency result $\cov(\SNwf)<\non(\SNwf)<\cof(\SNwf)$ is originally due to the first author~\cite{cardona} via a forcing matrix iteration construction. To force a value of $\cof(\SNwf)$, he used the matrix structure of the iteration to construct a suitable directed dominating system. Thanks to the results in \autoref{BounSN}, there is no need to construct a directed dominating system and the proof becomes more direct. We illustrate this new proof as an example.

\begin{theorem}[{\cite[Thm.~4.6]{cardona}}]\label{mainMigue}
Assume that $\aleph_1\leq\kappa \leq\lambda=\lambda^{<\lambda}$ are regular cardinals, 
and let $\lambda_1$ and $\lambda_2$ be infinite cardinals such that
$\lambda_1=\lambda_1^{\aleph_0}$, and $\lambda_2=\lambda_2^\lambda$. Then there is a cofinality preserving poset that forces:
\begin{enumerate}[label = \rm (\alph*)]
    \item\label{Migue1} $\add(\Nwf)=\non(\Mwf)=\kappa$ and $\cov(\Mwf)=\cof(\Nwf)=\lambda$.
    \item\label{Migue2} $\add(\SNwf)=\cov(\SNwf)=\kappa\leq\non(\SNwf)=\lambda < \cof(\SNwf)=\dfrak_{\lambda}=\lambda_2$.
    \item\label{Migue3} $\cfrak=\lambda_1$.
\end{enumerate}
\end{theorem}

One advantage of this construction over the models of $\cov(\SNwf)<\non(\SNwf)<\cof(\SNwf)$ from \autoref{Cmodel} and~\ref{Dmodel} is that $\cfrak=\lambda_1$ can be forced in any position with respect to $\cof(\SNwf)=\lambda_2$ (since we can assume in the ground model one of $\lambda_1<\lambda_2$, $\lambda_1= \lambda_2$ and $\lambda_2<\lambda_1$), while in the models from the corollaries we only get $\cfrak<\cof(\SNwf)$.

\begin{proof}
    We first force with $\Fn_{<\lambda}(\lambda_2,2)$ to get, in the generic extension, \[\dfrak_\lambda=\cof\left(([\lambda]^{<\kappa})^\lambda\right)=2^\lambda=\lambda_2.\] 
    Note that the cardinal arithmetic hypotheses still hold.

    Afterwards, force with $\Cor_{\lambda_1}\ast\Qnm$, where $\Qnm$ is a name for the ccc poset from~\cite[Thm.~11]{mejiamatrix}. In the final generic extension,~\ref{Migue1} and~\ref{Migue3} hold, and, since the latter poset is ccc, we get $\dfrak_{\lambda}=\cof\left(([\lambda]^{<\kappa})^\lambda\right)=\lambda_2$ by \autoref{lem:prcof}. Since $\cov(\Mwf)=\dfrak =\lambda$, we obtain by \autoref{cor:lowSN} and~\ref{new_upperb} that
    $\dfrak_\lambda\leq \cof(\SNwf) \leq \cov\left(([\lambda]^{<\kappa})^\lambda\right)$, so these three cardinals equal $\lambda_2$.

    The equality $\non(\SNwf)=\lambda$ is clear because it is between $\cov(\Mwf)$ and $\non(\Nwf)$. On the other hand, $\kappa = \add(\Nwf)\leq \add(\SNwf)$, and $\cov(\SNwf)\leq \kappa$ by \autoref{upbcov} because $\Cor_{\lambda_1}\ast\Qnm$ is a FS iteration of length with cofinality $\kappa$.
\end{proof}

The previous theorem is restricted to $\lambda$ regular, but thanks to the results of \autoref{BounSN} and~\ref{CovSN}, we can force the same result with $\lambda$ possibly singular. This is the main result~\autoref{mainappl} presented in the introduction. 

\begin{theorem}\label{Cmainth}
Assume that $\kappa\geq\aleph_1$ is a regular cardinal, $\kappa\leq\lambda=\cof([\lambda]^{<\kappa})\leq\lambda_1=\lambda_1^{\aleph_0}$, $\cf(\lambda)^{<\cf(\lambda)}=\cf(\lambda)$ and $\lambda<\lambda_2=\lambda_2^\lambda$. Then there is a cofinality preserving poset that forces:
\begin{enumerate}[label = \rm (\alph*)]
    \item\label{Cmaintha} $\add(\Nwf)=\non(\Mwf)=\kappa\textrm{\ and\ }\cov(\Mwf)=\cof(\Nwf)=\lambda$.
    \item\label{Cmainthb} $\add(\SNwf)=\cov(\SNwf)=\kappa\leq\non(\SNwf)=\lambda < \cof(\SNwf)=\dfrak_{\cf(\lambda)}=\lambda_2$.
    \item\label{Cmainthc} $\cfrak=\lambda_1$.
\end{enumerate}
\end{theorem}
\begin{proof} 

First note that $\cof([\lambda]^{<\kappa})=\lambda$ implies that $\cf(\lambda)\geq \kappa$, see \autoref{ex:ideal<theta}.

Start by forcing with $\Por_0=\Fn_{<\cf(\lambda)}(\lambda_2,2)$. In the generic extension, $\dfrak_{\cf(\lambda)} =\dfrak_\lambda =\cov\left(([\lambda]^{<\kappa})^\lambda\right) = \cof\left(([\lambda]^{<\kappa})^\lambda\right) = 2^\lambda=\lambda_2$, and the cardinal arithmetic hypotheses still hold (because $\Por_0$ is $\cf(\lambda)$-closed).


Afterwards, force with $\Cor_{\lambda_1}$. Then, in the generic extension, $\cfrak = \lambda_1$ and, by \autoref{lem:prcof}, $\dfrak_{\cf(\lambda)}= \dfrak_\lambda = \cov\left(([\lambda]^{<\kappa})^\lambda\right) = \cof\left(([\lambda]^{<\kappa})^\lambda\right) = \lambda_2$ and $\cof([\lambda]^{<\kappa}) = \lambda$ is preserved from $V^{\Por_0}$.

Work in $V_{0,0}:=V^{\Por_0\ast\Cor_{\lambda_1}}$, and force with the poset from~\cite[Thm.~4.6~(e)]{mejiavert}, which we call $\Q$. We review how $\Q$ is constructed and its main properties. This forcing is  a matrix iteration with vertical support restrictions, i.e.\ of the form $\la\Por_{A,\xi}:\, A\subseteq\lambda,\ \xi\leq \lambda\ra$ such that
\begin{enumerate}[label = \rm(M\arabic*)]
    \item\label{M1} For each $A\subseteq\lambda$, $\la\Por_{A,\xi}:\, \xi\leq\lambda\ra$ is constructed as a finite support iteration (so $\Por_{A,0}=\{\la\ \ra\}$ and direct limits are taken when $\xi$ is limit),
    \item $\Por_{A,1} := \Fn_{<\omega}(A,2)$, and
    \item\label{M3} for each $0<\xi<\lambda$, a set $\Delta(\xi)\in[\lambda]^{<\kappa}$ is chosen, and $\Por_{A,\xi+1}=\Por_{A,\xi}\ast\Qnm^A_\xi$ where
    \[\Qnm^A_\xi := \left\{
       \begin{array}{ll}
         \Aor^{V_{\Delta(\xi),\xi}} & \text{if $\Delta(\xi)\subseteq A$,}\\
         \{ \emptyset \} & \text{otherwise.}
       \end{array}
    \right.\]
    Here, $\Aor$ denotes amoeba forcing, and $V_{A,\xi}$ is the $\Por_{A,\xi}$-extension of $V_{0,0}$. 
    \item\label{M4} The sets $\la\Delta(\xi):\, 0<\xi<\lambda\ra$ are chosen in such a way that, for any $A\in[\lambda]^{<\kappa}$, $\set{0<\xi<\lambda}{A\subseteq \Delta(\xi)}$ is cofinal in $\lambda$ (which is possible by the hypothesis $\cof([\lambda]^{<\kappa}) =\lambda$.)
\end{enumerate}
The theory of matrix iterations with vertical support restrictions from~\cite{mejiavert} allows to construct such an iteration by recursion on $\xi\leq\lambda$, and it is guaranteed that $\Por_{A,\xi}$ is a complete subposet of $\Por_{B,\eta}$ for all $A\subseteq B\subseteq \lambda$ and $\xi\leq\eta$. Moreover, this theory guarantees:
\begin{enumerate}[label = \rm (V\arabic*)]
    \item\label{V0} Each $\Por_{A,\xi}$ has ccc.
    \item\label{V1} If $B\subseteq\lambda$, $\xi\leq\lambda$, and $\dot x$ is a (nice) $\Por_{B,\xi}$-name of a real, then there is some $A\in [B]^{<\kappa}$ such that $\dot x$ is a $\Por_{A,\xi}$-name.
    \item\label{V2} Denote by $\la\dot c_{\alpha}:\, \alpha<\lambda\ra$ the Cohen reals added by $\Por_{\lambda,1} = \Cor_\lambda$. If $A\subseteq B\subseteq\lambda$, $0<\xi\leq\lambda$, and $\alpha\in B\menos A$, then $\Por_{B,\xi}$ forces that $\dot c_\alpha$ is Cohen over $V_{A,\xi}$. 
\end{enumerate}
We define $\Qor:=\Por_{\lambda,\lambda}$. We show that $\Qor$ forces~\ref{Cmaintha}--\ref{Cmainthc}. Property~\ref{Cmainthc} is forced because it can be proved, by recursion on $\xi\leq\lambda$, that $|\Por_{A,\xi}|\leq \lambda_1$ for all $A\subseteq \lambda$.

To force~\ref{Cmaintha}, it is enough to force:\footnote{Note that $[\lambda]^{<\kappa} \eqT \Cbf_{[\lambda]^{<\kappa}}$ by \autoref{goodcof}.}
\begin{enumerate}[label = \rm (F\arabic*)]
    \item\label{F1} $\Nwf\leqT \lambda\times[\lambda]^{<\kappa}$, which implies by \autoref{FacprodRS} that
    \[\cof(\Nwf)\leq \cf(\lambda)\cdot \cof([\lambda]^{<\kappa}) =\lambda \text{ and } \kappa = \min\{\cf(\lambda),\add([\lambda]^{<\kappa})\}\leq \add(\Nwf).\]
    \item\label{F2} $\Cbf_{[\lambda]^{<\kappa}}\leqT \Cbf_\Mwf$, which implies 
    \[\lambda = \cov([\lambda]^{<\kappa}) \leq\cov(\Mwf) \text{ and } \non(\Mwf)\leq \non([\lambda]^{<\kappa}) = \kappa.\]
\end{enumerate}
Work in $V_{\lambda,\lambda}$. To show~\ref{F1}: if $N$ is a Borel measure zero set coded in $V_{\lambda,\lambda}$ then, by~\ref{V1}, there is some $A_N\in[\lambda]^{<\kappa}$ such that $N$ is coded in $V_{A_N,\lambda}$, and by~\ref{M1}, there is some $\xi_N<\lambda$ such that $N$ is coded in $V_{A_n,\xi_N}$; if $(\eta,B)\in \lambda\times[\lambda]^{<\kappa}$ then, by~\ref{M4}, there is some $\eta\leq \eta'<\lambda$ such that $B\subseteq \Delta(\eta')$, so let $N'_{\eta,B}$ be the generic amoeba real added by $\Aor^{V_{\Delta(\eta'),\eta'}}$. Note that $N\mapsto (\xi_N,A_N)$ and $(\eta,B)\mapsto N'_{\eta,b}$ is the required Tukey connection.

To show~\ref{F2}: if $M$ is a Borel meager set coded in $V_{\lambda,\lambda}$ then, by~\ref{V1}, there is some $B_M\in[\lambda]^{<\kappa}$ such that $M$ is coded in $V_{B_M,\lambda}$. Hence, the maps $\alpha \mapsto c_\alpha$ (the Cohen reals from~\ref{V2}) and $M\mapsto B_M$ form the required Tukey connection by~\ref{V2}.

Additionally, by \autoref{lem:prcof}, $\dfrak_{\cf(\lambda)}= \dfrak_\lambda = \cov\left(([\lambda]^{<\kappa})^\lambda\right) = \cof\left(([\lambda]^{<\kappa})^\lambda\right) = \lambda_2$ is preserved from $V_{0,0}$. On the other hand, since $\cov(\Mwf)=\dfrak=\non(\SNwf)=\supcof=\lambda$, by~\autoref{cor:lowSN} $\lambda_2=\dfrak_{\cf(\lambda)}\leq\cof(\SNwf)$, and $\cof(\SNwf)\leq\cov\left(([\lambda]^{<\kappa})^\lambda\right)=\lambda_2$ by ~\autoref{new_upperb}.

To establish $\add(\SNwf)=\cov(\SNwf)=\kappa$ note that, by~\ref{Cmaintha}, $\kappa =\add(\Nwf)\leq\add(\SNwf)$, so it remains to prove that $\cov(\SNwf)\leq\kappa$. 

It is enough to show that, in $V_{0,0}$, $\Por_{\lambda,\lambda}$ forces that $\Cbf_{[\lambda]^{<\kappa}}\leqT \Cbf^\perp_\SNwf$. For this purpose, partition $\lambda$ into sets $\la L_\zeta:\, \zeta<\lambda\ra$ of size $\lambda$. Thanks to~\ref{V2}, for each $\zeta<\lambda$ we can apply \autoref{lemcovp} to the sequence $\tbf^\zeta:= \la\Por_{A^\zeta_{\alpha},\lambda}:\, \alpha\in L_\zeta\ra$, where $A^\zeta_{\alpha}:=(\lambda\menos L_\zeta)\cup (L_\zeta\cap\alpha)$, to get that $\Por_{\lambda,\lambda}$ forces
$S_\zeta:=2^\omega\cap V_{\lambda\menos L_\zeta,\lambda}$ with strong measure zero. In $V_{\lambda,\lambda}$, consider the maps $\Psi_-\colon \lambda\to \SNwf$ and $\Psi_+\colon 2^\omega\to [\lambda]^{<\kappa}$ defined by $\Psi_-(\zeta):= S_\zeta$ and $\Psi_+(x):= \set{\zeta<\lambda}{x\notin S_\zeta}$. To see that indeed $|\Psi_+(x)| < \kappa$, by~\ref{V1} choose
$A_x\in[\lambda]^{<\kappa}$ such that $x\in V_{A_x,\lambda}$, and note that $x\in S_\zeta$ for any $\zeta<\lambda$ such that $A_x\cap L_\zeta = \emptyset$, hence $\Psi_+(x)\subseteq \set{\zeta<\lambda}{A_x\cap L_\zeta \neq \emptyset}$, which has size ${\leq} |A_x|<\kappa$. It is clear that $(\Psi_-,\Psi_+)$ is the required Tukey connection.
\end{proof}

\section{Open questions}\label{sec:Oq}

There are several questions about the principle $\DS(\delta)$.

\begin{question}\label{QDS}
    Does $\thzfc$ prove that $\DS(\delta)$ holds for some $\delta$?
\end{question}

We believe that a negative answer to this question would be related to BC (Borel's Conjecture). We do not know whether $\DS(\delta)$ is false for all $\delta$ in Laver model or in Mathias model, but if it is true in some of these models, then so is $\DS(\omega_2)$ (by \autoref{DSb=d}). So, to check that $\DS(\delta)$ is false for all $\delta$ in any of these models, it is enough to check that $\DS(\omega_2)$ is false.

Denote by $\delta_0$ the smallest ordinal $\delta$ satisfying $\DS(\delta)$ (in case it exists). By \autoref{lem:ondelta}, $\delta_0$ is additive indecomposable and $|\delta_0|=\dfrak$.

\begin{question}\label{QDSd}
    Is $\delta_0=\dfrak$ (if $\delta_0$ exists)?
\end{question}

In connection with \autoref{optimal}~\ref{delta0<cof}, we also ask:

\begin{question}\label{QDSnon}
\ 
\begin{enumerate}[label = \rm (\arabic*)]
    \item\label{d0<cof} Is $\delta_0\leq\cof(\SNwf)$ (if $\delta_0$ exists)?
    \item\label{cf<non} Is $\cf(\delta)\leq\non(\SNwf)$ for all $\delta$ satisfying $\DS(\delta)$?
\end{enumerate}
\end{question}

There is a connection between \autoref{QDS} and~\ref{QDSnon}~\ref{cf<non}: $\DS(\delta)$ and $\cf(\delta)\leq\non(\SNwf)$ imply $\bfrak\leq \non(\SNwf)$ by \autoref{lem:ondelta}~\ref{it:cfdelta}, but the latter is not provable in ZFC. Therefore, if the answer to \autoref{QDSnon}~\ref{cf<non} is positive, then $\non(\SNwf)<\bfrak$ implies that $\DS(\delta)$ does not hold for any $\delta$, which would be a negative answer to \autoref{QDS}. So we also ask:

\begin{question}
    Does $\non(\SNwf)<\bfrak$ imply that $\DS(\delta)$ does not hold for any $\delta$?
\end{question}

Recall that BC implies $\non(\SNwf)=\aleph_1<\bfrak$ (due to Rothberger, see~\cite[Thm.~0.4]{JSW}).

In \autoref{optimal} we checked how optimal are the hypotheses of the main results in \autoref{BounSN}. However, we do not know the answer to the following:

\begin{question}
    Is the equality $\non(\SNwf)=\supcof$ really required in \autoref{lowerSN}?
\end{question}

Concerning $\cof(\SNwf)$, we believe that it has more lower bounds than those illustrated in \autoref{Cichonwith_SN}.

\begin{question}\label{Qlowcf}
    Which of $\bfrak$, $\dfrak$, $\supcof$, $\non(\Nwf)$ and $\cof(\Nwf)$ are lower bounds of $\cof(\SNwf)$?
\end{question}

Note the connection between ``$\dfrak\leq\cof(\SNwf)$" and \autoref{QDSnon}~\ref{d0<cof}.

In case that $\dfrak\leq\cof(\SNwf)$ is true in ZFC, we will have $\cov(\Mwf)<\cof(\SNwf)$ as a consequence of \autoref{largecfSN}. This would answer the following open problem:

\begin{question}[Yorioka~{\cite{Yorioka}}]
    Does $\thzfc$ prove $\aleph_1<\cof(\SNwf)$?
\end{question}

In connection with \autoref{Qlowcf}, we recall the following problem:

\begin{question}[{\cite[Q.~6.6]{CM}}]
    Does $\thzfc$ prove $\supcof=\cof(\Nwf)$ and $\minadd = \add(\Nwf)$?
\end{question}

From \autoref{mainMigue} and~\ref{Cmainth}, it is natural to ask whether these results can be improved to force $\add(\SNwf)<\cov(\SNwf)<\non(\SNwf) < \cof(\SNwf)$, which answers~\cite[(Q3), p.~303]{cardona}. The authors, with J.~Brendle, constructed a FS iteration forcing this~\cite{BCM2}.

{\small
\bibliography{bibli}
\bibliographystyle{alpha}
}

\end{document}